\documentclass[10pt,english]{article}
\usepackage[dvipsnames]{xcolor}
\usepackage{bbm}
\usepackage{fullpage}
\usepackage{tabularx,amsmath,amssymb,amsthm,graphicx,paralist,subfigure,pdfpages,algorithm,algpseudocode,paralist,multirow,comment,appendix,caption}
\usepackage[bmargin=1in, tmargin=1in, lmargin=1in, rmargin=1in]{geometry}
\usepackage[dvipsnames]{xcolor}
\usepackage[autostyle]{csquotes}
\usepackage[colorlinks]{hyperref}
\hypersetup{
citecolor = Blue,
linkcolor = CMU_red
}

\usepackage[noadjust]{cite}
\usepackage{verbatim}
\usepackage{mathtools}

\usepackage{makecell}
\usepackage{pifont}
\newcommand{\cmark}{\text{\ding{51}}}
\newcommand{\xmark}{\text{\ding{55}}}

\newtheorem{lem}{Lemma} 
\newtheorem{theorem}{Theorem}
\newtheorem{defi}{Definition}

\newtheorem{prop}{Proposition}

\newtheorem{rmk}{Remark}
\newtheorem{assump}{Assumption}

\definecolor{CMU_red}{RGB}{187,0,0}
\definecolor{CMU_gray}{RGB}{224,224,224}
\definecolor{CMU_dark_gray}{RGB}{102,102,102}
\definecolor{CMU_gold}{RGB}{170,102,0}
\definecolor{CMU_teal}{RGB}{0,102,119}
\definecolor{CMU_blue}{RGB}{34, 68, 119}
\definecolor{CMU_green}{RGB}{0, 136, 85}
\definecolor{CMU_dark_green}{RGB}{34, 68, 51}

\def\mbb{\mathbb}
\def\mb{\mathbf}
\def\mc{\mathcal}

\def\wh{\widehat}

\def\ol{\overline}
\def\ul{\underline}
\def\dom{\mbox{dom}}
\def\bds{\boldsymbol}

\def\ra{\rightarrow}
\def\R{\mbb{R}}
\def\a{\alpha}
\def\E{\mbb{E}}
\def\P{\mbb{P}}
\def\F{\mc{F}}
\def\s{\mb{s}}
\def\l{\left\langle}
\def\r{\right\rangle}
\def\n{\nabla}
\def\x{\mb x}
\def\y{\mb y}

\def\J{\mb J}
\def\I{\mb I}
\def\X{\bds\xi}
\def\z{\mb z}
\def\i{\mathbbm{1}}
\def\u{\mb u}
\def\bxi{\bds\xi}
\def\nf{\n\mb{f}}
\def\W{\mb W}
\def\bl{\left\langle}
\def\br{\right\rangle}
\def\argmin{\mbox{argmin}}
\def\p{\normalfont\text{prox}}
\def\bp{\textbf{\mbox{prox}}_{\alpha h}}

\def\dom{\mbox{dom}}

\def\gp{\mb{s}}
\def\sgp{\mb{g}}
\def\ltrue{\lambda_*}
\def\Wtrue{\W_*}
\def\vp{\varphi}

\begin{document}
\title{{A Stochastic Proximal Gradient Framework for  Decentralized Non-Convex Composite Optimization: Topology-Independent Sample Complexity and Communication Efficiency}
\thanks{The work of RX and SK has been partially supported by NSF under award \#1513936. The work of UAK has been partially supported by NSF under awards \#1903972 and \#1935555.}}
\author{
Ran Xin\thanks{Department of Electrical and Computer Engineering (ECE), Carnegie Mellon University, Pittsburgh, PA 15213, USA; Emails: \texttt{\{ranx,soummyak\}@andrew.cmu.edu}.}   
\qquad
Subhro Das\thanks{MIT-IBM Watson AI Lab, IBM Research, Cambridge, MA 02142, USA; Email:
\texttt{Subhro.Das@ibm.com}} 
\qquad
Usman A. Khan\thanks{Department of ECE, Tufts University, Medford, MA 02155, USA; Email: \texttt{khan@ece.tufts.edu}.}       \qquad
Soummya Kar\footnotemark[2]
}

\date{}
\maketitle

\begin{abstract}
Decentralized optimization is a promising parallel computation paradigm for large-scale data analytics and machine learning problems defined over a network of nodes. This paper is concerned with decentralized non-convex composite problems with population or empirical risk. In particular, the networked nodes are tasked to find an approximate stationary point of the average of local, smooth, possibly non-convex risk functions plus a possibly non-differentiable extended valued convex regularizer. Under this general formulation, we propose the first provably efficient, stochastic proximal gradient framework, called \texttt{ProxGT}.
Specifically, we construct and analyze several instances of \texttt{ProxGT} that are tailored respectively for different problem classes of interest. Remarkably, we show that the sample complexities of these instances are network topology-independent and achieve linear speedups compared to that of the corresponding centralized optimal methods implemented on a single node. 

\end{abstract}

{
  \hypersetup{linkcolor=black}
  \tableofcontents
}

\section{Introduction}
Decentralized optimization \cite{DGD_nedic}, also known as distributed optimization over graphs, is a general parallel computation model for minimizing a sum of cost functions distributed over a network of nodes without a central coordinator. This cooperative minimization paradigm, built upon local communication and computation, has numerous applications in estimation, control, adaptation, and learning problems that frequently arise in multi-agent systems~\cite{DGD_Kar,dopt_control_nedich,diffusion_Chen,SPM_Xin}. In particular, the sparse and localized peer-to-peer information exchange pattern in decentralized networks substantially reduces the communication overhead on the parameter server in the centralized networks, thus making decentralized optimization algorithms especially appealing in large-scale data analytics and machine learning tasks~\cite{DSGD_nips,SGP_ICML,DLam_Yuan}.  

\subsection{Background}
In this paper, we study the following decentralized \emph{non-convex composite} optimization problem defined over a network of~$n$ nodes:
\begin{align}\label{P}
\min_{\x\in\R^p}\Psi(\x) 
:= F(\x) + h(\x)
\qquad
\text{such that}
\qquad
F(x) := \frac{1}{n}\sum_{i=1}^n f_i(\x).
\end{align}
Here, each~$f_i:\R^p\ra\R$ is $L$-smooth, possibly non-convex, and is only locally accessible by node~$i$, while $h:\R^p\ra\R\cup\{+\infty\}$ is convex, possibly non-differentiable, and is commonly known by all nodes. Each~$f_i$ is a cost function associated with local data at node~$i$ while~$h$ serves as a regularization term that is often used to impose additional problem structures such as convex constraints\footnote{The domain of~$h$ acts effectively as the constraint set of Problem~\eqref{P}.} and/or sparsity; simple examples of~$h$ include the~$\ell_1$-norm or the indicator function of a convex set. 
The communication over the networked nodes is abstracted as a directed graph~${\mc{G} := (\mc{V},\mc{E})}$, where~${\mc{V} := \{1,\cdots,n\}}$ denotes the set of node indices and~${\mc{E}}\subseteq\mc{V}\times\mc{V}$ collects ordered pairs~$(i,r)$,~${i,r \in \mc{V}}$, such that node~$r$ sends information to node~$i$. We adopt the convention that~$(i,i)\in\mc{E},\forall i\in\mc{V}$.
Our focus in this paper is on the following formulations of the local costs~$\{f_i\}_{i=1}^n$ that commonly appear in the context of statistical learning~\cite{SP_book}:
\begin{itemize}

\item \textbf{Population risk:} In this case, each~$f_i$ in Problem~\eqref{P} is defined as
\begin{align}\label{P_online}
f_i(\x) := \E_{\bxi_i\sim\mc{D}_i}[G_i(\x,\bxi_{i})],
\end{align}
where~$\bxi_{i}$ is a random data vector supported on~$\Xi_i\subseteq\R^q$ with some unknown probability distribution~$\mc{D}_i$ and~$G_i:\R^p\times\R^q\ra\R$ is a Borel function. The stochastic formulation~\eqref{P_online} often corresponds to \emph{online} scenarios such that samples are generated from the underlying data stream in real time at each node~$i$, in order to construct stochastic approximation of~$\n f_i$ for the subsequent optimization procedure~\cite{SA}.

\item \textbf{Empirical risk:} As a special case of~\eqref{P_online}, i.e., when~$\bxi_i$ has a finite support set~$\Xi_i:=\{\bxi_{i,(1)},\cdots,\bxi_{i,(m)}\}$ for some~$m\geq1$, each~$f_i$ takes the deterministic form of
\begin{align}\label{P_offline}
f_i(\x) := \frac{1}{m}\sum_{s=1}^m G_i\big(\x,\bxi_{i,(s)}\big).    
\end{align}
As an alternate viewpoint, the formulation~\eqref{P_offline} may be considered as the sample average approximation of~\eqref{P_online}, where $\{\bxi_{i,(1)},\cdots,\bxi_{i,(m)}\}$ take the role of \emph{offline} samples generated from the distribution~$\mc{D}_i$~\cite{SP_book}. 
We are interested in modern-day big-data scenarios, where~$m$ is very large and thus stochastic gradient methods are often preferable over exact gradient ones that use the entire local data per update.

\end{itemize}

The above formulations are quite general and has found applications in, e.g., sparse non-convex linear models~\cite{LR_NCVX}, principle component analysis~\cite{MP_scutari}, and matrix factorization~\cite{MF_chi}. Our goal in this paper is thus on the design and analysis of efficient decentralized \emph{stochastic} gradient algorithms to find an $\epsilon$-stationary point of the global \emph{composite} function~$\Psi$ in Problem~\eqref{P} under both population risk~\eqref{P_online} and empirical risk~\eqref{P_offline}. 

\subsection{Literature Review}
The last decade has witnessed a growing research interest and literature in the area of decentralized optimization. For convex (composite) problems, we refer the readers to, e.g.,~\cite{PD_Xu,PGD_Yuan,PD_Lu,PGD_Yan,Prox_GT_Di_scutari,Network-DANE,GTVR,SPM_Xin,free_bits,acc_pd_xu,AccDVR,AB_Xin,PushPull_Pu,tutorial_nedich,PIEEE_Xin,gradslide_lan} and the references therein, for unifying frameworks and connections between the existing methods. On the other hand, the work on decentralized methods for non-convex composite problems is fairly limited. In the following, we review the existing results that are closely related to Problem~\eqref{P} under either~\eqref{P_online} or~\eqref{P_offline}. 

\vspace{0.2cm}
\noindent
\textbf{Decentralized stochastic smooth non-convex optimization.} The special case~$h(\x) = 0$ of~Problem~\eqref{P} has been relatively well-studied. For this smooth formulation, variants of decentralized stochastic gradient descent (\texttt{DSGD}), e.g., \cite{DSGD_nips,SGP_ICML,DLam_Yuan,DSGD_vlaski_2}, admit simple implementations yet provide competitive practical performance against centralized methods in homogeneous environments like data centers. When the data distributions across the network become heterogeneous, the performance of \texttt{DSGD} in both practice and theory degrades significantly~\cite{GTVR,SPM_Xin,SED_Pu,SED2_Yuan,DSGD_Pu}. To address this issue, stochastic methods that are robust to heterogeneous data have been proposed, e.g., \texttt{D2}~\cite{D2} that is derived from primal-dual formulations~\cite{EXTRA,EXTRA_revisit,NIDS,SED} and \texttt{GT-DSGD}~\cite{improved_DSGT_Xin,GNSD} that is based on gradient tracking~\cite{NEXT_scutari,GT_CDC,MP_Pu,harnessing,DIGing}. Reference~\cite{optimal_dsgt}
provides a lower bound for a class of population risk minimization problems
and shows that combining \texttt{GT-DSGD} with multi-round accelerated consensus attains this lower bound. Built on top of \texttt{D2} and \texttt{GT-DSGD}, recent work, e.g.,~\cite{GTSAGA_NCVX,GT-SARAH,D_Get,D-SPIDER-SFO,HSGD_Xin} further leverages variance reduction techniques~\cite{spiderboost,spider,HSARAH} to achieve accelerated convergence. These results are certainly promising, but they are not applicable to the general composite case where~$h(\x) \neq 0$. 

\vspace{0.2cm}
\noindent
\textbf{Decentralized non-convex composite optimization.} The general case~$h(\x) \neq 0$ in Problem~\eqref{P} is significantly less explored in the existing literature. The first algorithmic framework for decentralized non-convex composite optimization is due to \cite{NEXT_scutari}, where $h$ is handled in a successive convex approximation scheme. Reference \cite{Prox_DGD} presents decentralized proximal gradient descent which tackles~$h$ via proximal mapping. These works~\cite{NEXT_scutari,Prox_DGD}, however, require the gradient of~$F$ and the subdifferential of~$h$ to be uniformly bounded. This bounded (sub)gradient assumption is later removed in~\cite{MP_scutari,PDP0_Luo}, where compression and directed graphs are also considered respectively. A decentralized Frank-Wolfe method is proposed in~\cite{GT_Wai} to 
handle the case where $h$ is an indicator function of a convex compact set. We note that the aforementioned results~\cite{NEXT_scutari,Prox_DGD,MP_scutari,GT_Wai,PDP0_Luo} are exact gradient methods, which are in general not applicable to the population risk~\eqref{P_online} and also may not be efficient in the empirical risk setting~\eqref{P_offline} when the local data size~$m$ is relatively large. Towards stochastic gradient methods, \cite{PDSGD_NCVX} analyzes a projected \texttt{DSGD} type method for problems with compact constraint set. Reference~\cite{DSGD_Brian} establishes the \emph{asymptotic} convergence of \texttt{DSGD} for a family of non-convex non-smooth functions that satisfy certain coersive property.  
A recent work~\cite{PDP_Luo} presents \texttt{SPPDM}, a decentralized stochastic proximal primal-dual method, and provides related convergence guarantees under the assumption that the epigraph of~$h$ is a polyhedral set.

\vspace{0.2cm}
To the best of knowledge, in the literature of decentralized optimization, \emph{there is no non-asymptotic sample and communication complexity results of stochastic gradient methods for the non-convex composite problem with a general convex non-differentiable regularizer $h$.} We address this gap in this paper. 

\subsection{Our Contributions}
We develop a unified stochastic proximal gradient tracking framework, called \texttt{ProxGT}, for designing and~analyzing decentralized methods for the general non-convex composite problem. \texttt{ProxGT} allows flexible construction of local gradient estimators, where a suitable one may be chosen in light of the underlying problem specifications and practical applications. 
We highlight several important aspects of \texttt{ProxGT} in the following.

\begin{itemize}
\item \emph{Algorithm construction:} For definiteness, we present three instantiations of \texttt{ProxGT}. 
For the general population risk, we develop \texttt{ProxGT-SA} by using the minibatch stochastic approximation technique~\cite{mirror_prox_Lan}. 
Leveraging \texttt{SARAH} type variance reduction schemes~\cite{sarah_ncvx,spiderboost,book_lan}, we next provide two alternate algorithms, named \texttt{ProxGT-SR-O} and
\texttt{ProxGT-SR-E}, for the population and empirical risk respectively that outperform \texttt{ProxGT-SA} when a mean-squared smoothness property holds~\cite{lowerbound_sgd}.

\item \emph{Complexity results:} We establish the sample and communication complexities of 
the proposed \texttt{ProxGT-SA}, \texttt{ProxGT-SR-O}, and
\texttt{ProxGT-SR-E} algorithms to find an $\epsilon$-stationary solution; see Table~\ref{comp_oc} for a summary. Remarkably, we show that their sample complexities at each node are network topology-independent and are $n$ times smaller than that of the centralized \emph{optimal} algorithms implemented on a single node for the corresponding problem classes. In other words, \texttt{ProxGT-SA}, \texttt{ProxGT-SR-O}, and \texttt{ProxGT-SR-E} achieve a topology-independent linear speedup compared to their respective optimal centralized counterparts.

\item \emph{Analysis techniques:} Our convergence analysis is developed in a unified manner and can be used to analyze other instances of the \texttt{ProxGT} framework. In particular, we establish a novel \emph{stochastic descent inequality} for the non-convex composite objective $\Psi$ and a new \emph{consensus error bound} by carefully handling the proximal mapping. These intermediate technical results are of independent interest and may be used in analyzing other decentralized stochastic proximal first-order methods for non-convex composite problems.

\item \textit{Special cases:} For the special case $h = 0$, \texttt{ProxGT-SR-E} and \texttt{ProxGT-SR-O} also constitute improvements over the state-of-the-art decentralized variance-reduced methods \texttt{GT-SARAH}~\cite{GT-SARAH} and \texttt{GT-HSGD}~\cite{HSGD_Xin} in the following sense. For the empirical risk, \texttt{GT-SARAH} attains the optimal centralized sample complexity when the local sample size $m$ is large enough. Similarly, for the population risk, \texttt{GT-HSGD} is optimal when the required accuracy is small enough. \texttt{ProxGT-SR-O} and \texttt{ProxGT-SR-E} improve these regime restrictions and attain the corresponding optimal centralized complexities by performing multiple rounds of (accelerated) consensus updates per iteration.

\end{itemize}

{\renewcommand{\arraystretch}{2}
\begin{table*}[!ht]
\footnotesize
\caption{A summary of the sample and communication complexities of the instances of \texttt{ProxGT} studied in this paper for finding an $\epsilon$-stationary point of the global composite function~$\Psi$. In the table,~$n$ is the number of the nodes,~$(1 - \ltrue)\in(0,1]$ is the spectral gap of the weight matrix associated with the network, $L$ is the smoothness parameter for the risk functions, $\Delta$ is the fucntion value gap, $\nu^2$ is the stochastic gradient variance under the expected risk, $m$ is the local sample size under the empirical risk. The MSS column indicates whether the algorithm in question requires the mean-squared smoothness assumption.}
\label{comp_oc}
\vspace{-0.5cm}
\begin{center}
\begin{tabular}{|c|c|c|c|c|r|}
\hline
\textbf{Algorithm} & \textbf{Sample Complexity at Each Node} & \textbf{Communication Complexity} & \textbf{MSS} & \textbf{Remarks}\\ \hline
\texttt{ProxGT-SA} & $\mc{O}\left(\dfrac{L\Delta\nu^2}{n\epsilon^4}\right)$ & $\mc{O}\left(\dfrac{L\Delta}{\epsilon^2}\cdot\dfrac{\log n}{\sqrt{1-\ltrue}}\right)$ & \xmark & Population Risk~\eqref{P_online} 
\\ \hline
\texttt{ProxGT-SR-O} & $\mc{O}\left(\dfrac{L\Delta\nu}{n\epsilon^3} + \dfrac{\nu^2}{n\epsilon^2}\right)$ & $\mc{O}\left(\left(\dfrac{L\Delta}{\epsilon^2}+\dfrac{\nu}{\epsilon}\right)\dfrac{\log n}{\sqrt{1-\ltrue}}\right)$ & \cmark & Population Risk~\eqref{P_online}  \\ \hline
\texttt{ProxGT-SR-E} &  $\mc{O}\left(\dfrac{L\Delta}{\epsilon^2}\max\left\{\sqrt{\dfrac{m}{n}},1\right\}+\max\left\{m,\sqrt{nm}\right\}\right)$ & $\mc{O}\left(\left(\dfrac{L\Delta}{\epsilon^2}+\sqrt{nm}\right)\dfrac{1}{\sqrt{1-\ltrue}}\right)$ & \cmark & Empirical Risk~\eqref{P_offline} \\ \hline
\end{tabular}
\end{center}
\end{table*}}

\subsection{Notation}
The set of positive real numbers is denoted by~$\R^+$. For an integer~$z\geq1$, we denote~$[z]:= \{1,\cdots,z\}$.
The floor and ceiling function are written as~$\lfloor\cdot\rfloor$ and~$\lceil\cdot\rceil$ respectively.
We use lowercase bold letters to denote vectors and uppercase bold letters to denote matrices.
The~$d\times d$ identity matrix is denoted by $\mb{I}_d$, while the~$d$-dimensional column vectors of all ones and zeros are represented by $\mb{1}_d$ and $\mb{0}_d$ respectively. For a matrix~$\mb{X}\in\R^{d\times d}$, its~$(i,r)$-th entry is denoted by~$[\mb{X}]_{i,r}$.
We use $\mb{X}\otimes \mb{Y}$ to denote the Kronecker product of two matrices~$\mb{X}$ and~$\mb{Y}$.
The Euclidean norm of a vector or the spectral norm of a matrix is denoted by~$\|\cdot\|$. 

For an extended valued function $h:\R^p\ra\R\cup\{+\infty\}$, we denote $\dom(h) := \{\x: h(\x) < + \infty\},$ 
and $h$ is said to be proper if $\dom(h)$ is nonempty. For~$\x\in\dom(h)$, we denote the subdifferential of $h$ at $\x$ by $\partial h(\x)$,~i.e.,
$\partial h(\x) := \big\{ \mb{u} : h(\y) \geq h(\x) + \l \mb{u}, \y - \x\r, \forall \y\in\dom(h) \big\}.$
The proximal mapping of~$h$ is defined as 
\begin{align}\label{def_prox}
\p_{h}(\x) := \argmin_{\mb{u}\in\R^p}\left\{\frac{1}{2}\|\mb{u} - \x\|^2 + h(\u)\right\}.    
\end{align}

We work with a rich enough probability triple~$(\Theta,\mc{F},\P)$, where all random objects are defined properly. Given a sub-$\sigma$-algebra~$\mc{H}\subseteq\mc{F}$ and a random vector~$\x$, we write~$\x\in\mc{H}$ if~$\x$ is~$\mc{H}$-measurable. We use $\sigma(\cdot)$ to denote the~$\sigma$-algebra generated by the argument events and/or random vectors. 

\subsection{Roadmap}
The remainder of the paper is organized as follows. Section~\ref{sec_asp} formulates the problems. Section~\ref{sec_alg} develops the proposed algorithmic framework and its instances of interest in this paper. Section~\ref{sec_main} 
presents the main convergence results of the proposed algorithms and discuss their implications. Section~\ref{sec_concl} concludes the paper.
Detailed proofs are provided in the appendix.

\section{Problem Formulation}\label{sec_asp}
\subsection{The Non-Convex Composite Model}
We make the following assumption on the objective functions.
\begin{assump}[\textbf{Functions}]\label{asp_f}
In Problem~\eqref{P}, the following statements hold:
\begin{enumerate}[(a)]
\item $h:\R^p\ra\R\cup\{+\infty\}$ is proper, closed, and convex;
\label{asp_f_r}
\item Each $f_i:\R^p\ra\R$ is~$L$-smooth, i.e.,~$\|\n f_i(\x)-\n f_i(\y)\|\leq L\|\x-\y\|$,~$\forall \x,\y\in\R^p$, for some~$L\in\R^+$;
\item $\Psi$ is bounded below, i.e.,~$\ul{\Psi}:=\inf_{\x\R^p}\Psi(\x)>-\infty$.
\end{enumerate}
\end{assump}
\noindent
Assumption~\ref{asp_f} describes the standard \emph{non-convex composite model}~\cite[Section 10.1]{book_Beck2}. A simple example of the extended real-valued function~$h$ that satisfies Assumption~\ref{asp_f}\eqref{asp_f_r} is the indicator of a nonempty, closed, and convex set in~$\R^p$. 
Under Assumption~\ref{asp_f}, we say that a point~$\wh{\x}\in\dom(h)$ is \emph{stationary} for Problem~\eqref{P} if
\begin{align}\label{def_stat}
-\n F(\wh{\x}) \in \partial h(\wh{\x}).
\end{align}

\begin{rmk}
It is shown in~\cite[Theorem 3.7.2]{book_Beck2} that the stationary condition~\eqref{def_stat} is a necessary condition for a point~$\wh{\x}$ to be a local optimal solution of Problem~\eqref{P}. 
\end{rmk}

Based on the definition of the proximal mapping~\eqref{def_prox}, it can be shown that this stationarity condition is equivalent to a fixed point equation, i.e.,~$\wh{\x}\in\R^p$ is stationary for Problem~\eqref{P} if and only if
\begin{align}\label{fix_point}
\wh{\x} = \p_{\alpha h}\big(\wh{\x} - \alpha \n F(\wh{\x})\big),
\qquad \forall \alpha > 0.
\end{align}
In view of~\eqref{fix_point}, we define the \emph{gradient mapping} for Problem~\eqref{P}:
\begin{align}\label{GP}
\gp(\x) := \frac{1}{\a}\Big(\x - \p_{\alpha h}\big(\x - \alpha \n F(\x)\big)\Big), \qquad\forall\x\in\dom(h),
\end{align}
where~$\a>0$. We note that the gradient mapping~$\gp(\x)$ can be viewed as a generalized gradient of~$\Psi$ at~$\x$ in the sense that~$\gp(\x) = \n F(\x)$ if~$h=0$. The size of~$\gp(\cdot)$ thus serves as a natural measure for the approximate stationarity of a solution~\cite{book_Beck2,book_lan}. 

\begin{defi}[\textbf{$\epsilon$-stationarity}]\label{def_eps_stat}
Let Assumption~\ref{asp_f} hold. Then a random vector~$\x\in\dom(h)$ is said to be an $\epsilon$-stationary solution for Problem~\eqref{P} if~$\E[\|\gp(\x)\|^2]\leq\epsilon^2$, where~$\s(\cdot)$ is defined in~\eqref{GP}.
\end{defi}

\subsection{The Network Model}
We make the following assumption on the directed graph~$\mc{G} = (\mc{V},\mc{E})$ which characterizes the decentralized communication between the networked nodes.

\begin{assump}[\textbf{Network}]\label{asp_net}
The directed network~$\mc{G} = (\mc{V},\mc{E})$ is strongly connected, i.e., there is a~directed path from each node to every other node. Moreover, there exists a doubly stochastic weight matrix $\Wtrue\in\R^{n\times n}$ associated with~$\mc{G}$, i.e., $\Wtrue$ satisfies the following conditions:
\begin{enumerate}[(a)]
\item $[\Wtrue]_{i,r}>0$ if~$(i,r)\in\mc{E}$ and $[\Wtrue]_{i,r} = 0$ if $(i,r)\notin\mc{E}$.
\item $\Wtrue\mb{1}_n = \Wtrue^\top\mb{1}_n = \mb{1}_n$.
\end{enumerate}
\end{assump}
Assumption~\ref{asp_net} describes the standard consensus weight matrix $\Wtrue$~\cite{tutorial_nedich}. Under this assumption, i.e., $\Wtrue$ is primitive and doubly stochastic, it is well-known that 
\begin{align}\label{lambda}
\ltrue := \left\| \Wtrue - \tfrac{1}{n}\mb{1}_n\mb{1}_n^\top \right\| \in [0,1),   
\end{align}
where $\ltrue$ is the second largest singular value of~$\Wtrue$ and characterizes the connectivity of the network~$\mc{G}$~\cite{matrix_analysis,tutorial_nedich}: if~$\mc{G}$ is fully connected then~$\ltrue = 0$; as the connectivity of~$\mc{G}$ becomes weaker,~$\ltrue$ approaches to~$1$. Therefore, we refer~$(1-\ltrue)\in(0,1]$ as the \emph{spectral gap} of the network~$\mc{G}$. When~$\mc{G}$ is undirected, $\Wtrue$ in Assumption~\ref{asp_net} always exists and can be constructed efficiently via exchange of local degree information between the node~\cite{tutorial_nedich}. Otherwise, $\Wtrue$ may be found by decentralized recursive algorithms provided that it exists~\cite{digraph}.

\subsection{Stochastic Gradient Models}
We make a blanket assumption that each node~$i$ at every iteration~$t$ is able to obtain i.i.d. minibatch samples $\{\bxi_{i,s}^t: s\in[b_t]\}$ for the local random data vector~$\bxi_i$. The induced natural filtration is given by
\begin{align}\label{Ft}
\F_t :=&~\sigma\big(\X_{i,s}^{r}: \forall i\in\mc{V}, s\in[b_r], 1\leq r \leq t-1\big), \qquad\forall t\geq2, \\
\F_1 :=&~\{\Theta,\phi\}. 
\nonumber
\end{align}
Intuitively, the filtration $\F_t$ represents the historical information of an algorithm that samples~$\X_i$ up to iteration~$t$. We require that the stochastic gradient $\n G(\cdot,\bxi_{i,s}^t)$ is conditionally unbiased with respect to~$\F_t$.
\begin{assump}[\textbf{Unbiasedness}]\label{asp_unbias}
$\E\big[\n G_i(\x,\X_{i,s}^{t})|\F_t\big] = \nabla f_i(\x)$, $\forall i\in\mc{V}$, $\forall t\geq1$, $\forall s\in[b_t]$, $\forall \x\in\F_t$.
\end{assump}

\begin{rmk}
Under the empirical risk~\eqref{P_offline}, Assumption~\ref{asp_unbias} amounts to uniform sampling at random from $[m]$.
\end{rmk}

We consider a standard bounded variance assumption~\cite{book_lan} for~$\n G_i(\cdot,\X_{i,s}^t)$.
\begin{assump}[\textbf{Bounded Variance}]\label{asp_bvr}
Let~$\nu_i\in\R^+$, $\forall i\in\mc{V}$.
We have $\E\big[\|\n G_i(\x,\X_{i,s}^{t}) - \n f_i(\x)\|^2|\F_t\big] \leq \nu_i^2$, $\forall t\geq1$, $\forall s\in[b_t]$, $\forall \x\in\F_t$, $\forall i\in\mc{V}$; $\nu^2 := \frac{1}{n}\sum_{i=1}^n\nu^2_i$.
\end{assump}

We are also interested in the case when the stochastic gradients further satisfy the mean-squared smoothness property which is often satisfied in machine learning models~\cite{lowerbound_sgd,book_lan}.

\begin{assump}[\textbf{Mean-Squared Smoothness}]\label{asp_mss}
Let~$L\in\R^{+}$. In the case of population risk~\eqref{P_online},
we have
\begin{align*}
\E\big[\|\n G_i(\x,\X_{i}) - \n G_i(\y,\X_{i})\|^2\big] 
\leq L^2\E\big[\|\x-\y\|^2\big],    
\end{align*}
for all $i\in\mc{V}$ and $\x,\y\in\R^p$.
In the case of empirical risk~\eqref{P_offline}, the above statement reduces to
$$\frac{1}{m}\sum_{s=1}^m\big\|\n G_i(\x,\X_{i,(s)}) - \n G_i(\x,\X_{i,(s)})\big\|^2 \leq L^2\|\x-\y\|^2,$$
for all $i\in\mc{V}$ and $\x,\y\in\R^p$.
\end{assump}

\begin{rmk}
Assumption~\ref{asp_mss} implies that each~$f_i$ is~$L$-smooth by Jensen's inequality.
\end{rmk}

\section{Algorithm Development}\label{sec_alg}
In the centralized scenarios, a popular fixed-point method to solve~\eqref{fix_point} is
the proximal gradient descent~\cite{book_Beck2}, which takes the following form:
\begin{align}\label{CPGD}
\x_{t+1} = \p_{\alpha h}\big(\x_t - \alpha \n F(\x_t)\big),
\qquad \forall t \geq 1,
\end{align}
where~$\a>0$. However, the recursion~\eqref{CPGD} cannot be directly implemented in a decentralized manner. The main challenge lies in the fact that the global gradient $\n F$ is not locally available at any node and cannot be computed via one-shot aggregation of local gradient information in decentralized networks. Moreover, the local gradients~$\{\n f_i\}_{i=1}^n$ are often significantly different due to the heterogeneous data across the nodes, making the classical gradient consensus approaches~\cite{tutorial_nedich} less effective especially in the non-convex settings \cite{GT-SARAH}. One popular technique to overcome these issues is gradient tracking~\cite{NEXT_scutari,GT_CDC}, which has been adopted in various decentralized stochastic gradient methods for smooth non-convex problems, e.g.,~\cite{improved_DSGT_Xin,HSGD_Xin,D_Get}. Inspired by these works, we propose a general proximal stochastic gradient tracking framework, termed as \texttt{ProxGT}, to tackle the non-convex non-smooth composite Problem~\eqref{P}.

\subsection{A Generic Algorithmic Procedure}
We now describe the proposed \texttt{ProxGT} framework. At every iteration~$t$, each node~$i$ in the network retains three local variables $\x_t^i$, $\mb{v}_t^i$, and $\y_t^i$, all in~$\R^p$, where~$\x_t^i$ approximates an stationary point of Problem~\eqref{P}, $\mb{v}_t^i$ estimates the local exact gradient~$\n f_i(\x_t^i)$ from the samples generated for~$\bxi_i$, and~$\y_t^i$ tracks the global gradient~$\n F(\x_t^i)$ via a stochastic gradient tracking type update~\cite{MP_Pu} from the local gradient estimates~$\{\mb{v}_t^i\}_{i=1}^n$. With the global gradient tracker~$\y_t^i$ at hand, each node~$i$ performs a local inexact fixed point update for~\eqref{fix_point}:
\begin{align*}
\z_{t+1}^i
:= \p_{\alpha h}\big(\x_{t}^i - \alpha \y_{t+1}^i\big),    
\end{align*}
where~$\a>0$ is the step-size. The local solution~$\x_{t+1}^i$ at the next iteration is then updated by performing consensus on the intermediate variables~$\{\z_{t+1}^i\}_{i=1}^n$ over the network. For the ease of presentation, we define 
\begin{align*}
\W := \Wtrue \otimes \I_p.
\end{align*}
and the global variables $\x_t$, $\mb{v}_t$, $\y_t$, $\z_t$ which concatenate their corresponding local variables, i.e.,
\begin{align*}
\x_t = 
\begin{bmatrix}
\x_t^1\\
\cdots \\
\x_t^n
\end{bmatrix},
\quad
\mb{v}_t = 
\begin{bmatrix}
\mb{v}_t^1\\
\cdots \\
\mb{v}_t^n
\end{bmatrix},
\quad
\y_t = 
\begin{bmatrix}
\y_t^1\\
\cdots \\
\y_t^n
\end{bmatrix},
\quad
\z_t = 
\begin{bmatrix}
\z_t^1\\
\cdots \\
\z_t^n
\end{bmatrix},
\end{align*}
all in~$\R^{np}$. With the help of these notations, we formally present \texttt{ProxGT} in Algorithm~\ref{PGT} from a global view. 




\begin{rmk}
In Algorithm~\ref{PGT}, the decentralized propagation and averaging of local variables over the network appear as matrix-vector products, while the node-wise implementation of \texttt{ProxGT} can be obtained accordingly. 
\end{rmk}

\begin{rmk}
We note that~$(\Wtrue\otimes\I_p)^K$ leads to~$K$ decentralized averaging step(s) over~$K$ rounds of communication in the corresponding update. This multi-consensus update with an appropriately chosen~$K$ is often helpful to achieve faster convergence in the corresponding algorithms~\cite{nesterov_jakovetic,Network-DANE,Prox_GT_Di_scutari}.
\end{rmk}

\begin{rmk}
The~$\x$- and~$\y$-updates in Algorithm~\ref{PGT} take the adapt-then-combine form~\cite{diffusion_Chen} which often leads improved practical performance.
\end{rmk}

\begin{algorithm}[tbph]
\caption{\textbf{\texttt{ProxGT}} for Problem~\eqref{P}}
\label{PGT}
\begin{algorithmic}[1]
\Require{$\x_{1} = \mb{1}_n\otimes \ol{\x}_1$;~$K$;~$\alpha$;~$\mb{y}_{1} = \mb{0}_{np}$; $\mb{v}_{0} = \mb{0}_{np}$.}
\vspace{0.1cm}
\For{$t = 1, \cdots, T$}
\vspace{0.1cm}
\State{Generate an estimator~$\mb{v}_t^i$ of $\n f_i(\x_t^i)$,~$\forall i$.}
\vspace{0.1cm}
\State{Tracking:
$\mb{y}_{t+1} 
= \W^K\big(\mb{y}_{t} + \mb{v}_t - \mb{v}_{t-1}\big).$}
\vspace{0.1cm}
\State{Prox-Descent:
$\z_{t+1}^i
= \p_{\alpha h}\big(\x_{t}^i - \alpha \y_{t+1}^i\big)$,~$\forall i$.}
\vspace{0.1cm}
\State{Consensus:
$\mb{x}_{t+1} = \W^K\z_{t+1}.$}
\vspace{0.1cm}
\EndFor
\end{algorithmic}
\end{algorithm}

\noindent
Before we proceed, it is helpful to further clarify the role of~$\y_t^i$ and the choice of~$\mb{v}_t^i$ in Algorithm~\ref{PGT}.

\begin{itemize}
\item With the help of the doubly stochasticity of the network weight matrix~$\Wtrue$, it is straightforward to show by induction that the~$\y$-update in Algorithm~\ref{PGT} satisfies an important dynamic tracking property:
\begin{align}\label{y_track}
\frac{1}{n}\sum_{i=1}^n\mb{y}_{t+1}^i 
= \frac{1}{n}\sum_{i=1}^n\mb{v}_t^i, \qquad\forall t\geq1.    
\end{align}
In view of~\eqref{y_track} and the recursion of Algorithm~\ref{PGT}, it is expected that each~$\mb{y}_t^i$ approaches~$\frac{1}{n}\sum_{i=1}^n\mb{v}_t^i$ and thus asymptotically tracks the global gradient~$\n F(\x_t^i)$.

\item Clearly, different choices of the gradient estimator~$\mb{v}_t^i$ lead to different instances of the \texttt{ProxGT} framework. Many local gradient estimation schemes are applicable here, such as the minibatch stochastic approximation~\cite{mirror_prox_Lan,SA} and various variance reduction schemes, e.g.,~\cite{SAGA_reddi,spider,spiderboost,book_lan,HSARAH}. As we explicitly show, a suitable choice can be made in light of the underlying problem class and practical application.
\end{itemize}

\subsection{Instances of Interest}\label{sec_alg_online}
In this section, we present several instances of \texttt{ProxGT} that are of particular interest for the population and empirical risk formulations considered in this paper. 

\subsubsection{Population Risk Minimization}
A natural choice of the gradient estimator~$\mb{v}_t^i$ in \texttt{ProxGT} is the minibatch stochastic approximation~\cite{mirror_prox_Lan}. The resulting instance, called \texttt{ProxGT-SA}, is presented in Algorithm~\ref{PGT_SA}.

\begin{algorithm}[tbph]
\caption{\textbf{\texttt{ProxGT-SA}} for Problem~\eqref{P} with~\eqref{P_online}}
\label{PGT_SA}
\begin{algorithmic}[1]
\Ensure{Replace Line 2 in Algorithm~\ref{PGT} by the following for all $i$.}
\Require{$b$.}
\vspace{0.1cm}
\State{Obtain i.i.d samples $\{\bxi_{i,s}^t: s\in[b]\}$ for~$\bxi_i$.} 
\vspace{0.05cm}
\State{Set $\mb{v}_t^i:= 
\frac{1}{b}\sum_{s=1}^b\n G_i(\x_t^i,\X_{i,s}^t)$, $\forall i$. }
\vspace{0.1cm}
\end{algorithmic}
\end{algorithm}

An alternate approach is to construct the gradient estimator~$\mb{v}_t^i$ in \texttt{ProxGT} via an online SARAH type recursive variance reduction scheme that effectively leverages the historical information to achieve faster convergence. When Assumption~\ref{asp_mss} holds, the resulting algorithm \texttt{ProxGT-SR-O} (given in Algorithm~\ref{PGT_SR_O}) shows superior performance over Algorithm~\ref{PGT_SA}.

\begin{algorithm}[tbph]
\caption{\textbf{\texttt{ProxGT-SR-O}} for Problem~\eqref{P} with~\eqref{P_online}}
\label{PGT_SR_O}
\begin{algorithmic}[1]
\Ensure{Replace Line 2 in Algorithm~\ref{PGT} by the following for all $i$.}
\Require{$B$, $b$, $q$.}
\If{$t \bmod q = 1$}
\vspace{0.1cm}
\State{Obtain i.i.d samples $\{\bxi_{i,s}^t: s\in[B]\}$ for~$\bxi_i$.}
\State{Set $\mb{v}_t^i:= 
\frac{1}{B}\sum_{s=1}^{B}\n G_i(\x_t^i,\X_{i,s}^t)$, $\forall i$. }
\Else
\vspace{0.1cm}
\State{Obtain i.i.d samples $\{\bxi_{i,s}^t: s\in[b]\}$ for~$\bxi_i$.}
\State{Set $\mb{v}_t^i:= 
\frac{1}{b}\sum_{s=1}^{b}\big(\n G_i(\x_t^i,\X_{i,s}^t) - \n G_i(\x_{t-1}^i,\X_{i,s}^t)\big) + \mb{v}_{t-1}^i$, $\forall i$. }
\vspace{0.1cm}
\EndIf
\end{algorithmic}
\end{algorithm}

\vspace{-0.6cm}
\subsubsection{Empirical Risk Minimization}
We now consider Problem~\eqref{P} under the empirical risk~\eqref{P_offline}, where the support of each local random data~$\bxi_i$ is a finite set~$\Xi_i =\{\bxi_{i,(1)},\cdots,\bxi_{i,(m)}\}$. We recall that~\eqref{P_offline}, in its essence, is a special case of the population risk~\eqref{P_online} and therefore \texttt{ProxGT-SA} and \texttt{ProxGT-SR-O} developed in Section~\ref{sec_alg_online} remain applicable. However, the finite-sum structure of each~$f_i$ under \eqref{P_offline} lends itself to faster stochastic variance reduction procedures~\cite{spider,spiderboost,sarah_ncvx}. In particular, we replace the periodic minibatch stochastic approximation step in \texttt{ProxGT-SR-O} by exact gradient computation. This corresponding implementation, named \texttt{ProxGT-SR-E}, is presented in Algorithm~\ref{PGT_SR_E}.

\begin{algorithm}[tbph]
\caption{\textbf{\texttt{ProxGT-SR-E}} for Problem~\eqref{P} with~\eqref{P_offline}}
\label{PGT_SR_E}
\begin{algorithmic}[1]
\Ensure{Replace Line 2 in Algorithm~\ref{PGT} by the following for all $i$.}
\Require{$b$, $q$.}
\If{$t \bmod q = 1$}
\vspace{0.1cm}
\State{Set $\mb{v}_t^i:= \n f_i(\x_t^i)$, $\forall i$. }
\Else
\vspace{0.1cm}
\State{Sample $\{\bxi_{i,s}^t: s\in[b]\}$ uniformly at random from~$\Xi_i$, $\forall i$.}
\State{Set $\mb{v}_t^i:= 
\frac{1}{b}\sum_{s=1}^{b}\big(\n G_i(\x_t^i,\X_{i,s}^t) - \n G_i(\x_{t-1}^i,\X_{i,s}^t)\big) + \mb{v}_{t-1}^i$, $\forall i$. }
\vspace{0.1cm}
\EndIf
\end{algorithmic}
\end{algorithm}

\vspace{-0.6cm}
\section{Main Results}\label{sec_main}
In this section, we present the main convergence results of the proposed algorithms and discuss their intrinsic features. Throughout the rest of this paper, we let Assumption~\ref{asp_f},~\ref{asp_net}, and~\ref{asp_unbias} hold without explicit statements. The iteration complexity of \texttt{ProxGT} and its instances is quantified in the following sense, while the sample and communication complexities can be obtained accordingly. 
\begin{defi}[\textbf{Iteration Complexity}]\label{def_metric}
Consider the random vectors~$\{\x_t^i\}$ generated by \texttt{ProxGT}. We say that \texttt{ProxGT} finds an~$\epsilon$-stationary point of Problem~\eqref{P} in~$T$ iterations if
\begin{align}\label{metric}
\frac{1}{T}\sum_{t=1}^{T}\frac{1}{n}\sum_{i=1}^n\E\Big[\big\|\gp\big(\x_t^i\big)\big\|^2 + L^2\big\|\x_t^i - \ol{\x}_t\big\|^2 \Big] \leq \epsilon^2,   
\end{align}
where~$\ol{\x}_t := \frac{1}{n}\sum_{i=1}^n\x_t^i$ and the gradient mapping~$\gp(\cdot)$ is defined in~\eqref{GP}.
\end{defi}
\noindent
In view of Definition~\ref{def_eps_stat}, this metric concerns the stationary gap and the consensus error over the network. In particular, if~\eqref{metric} holds and we select the output, say~$\wh{\x}$, of \texttt{ProxGT} uniformly at random from $\{\x_t^i: t\in[T], i\in\mc{V}\}$, then~$\E[\|\gp(\wh{\x})\|^2]\leq\epsilon^2$, i.e.,~$\wh{\x}$ is an $\epsilon$-stationary solution for Problem~\eqref{P}. 

\subsection{Sample and Communication Complexity Results}
For ease of presentation, we define
\begin{align}\label{delta_zeta}
\Delta := \Psi(\ol{\x}_1) - \ul{\Psi}
\qquad
\text{and}
\qquad
\zeta^2 := \frac{1}{n}\sum_{i=1}^n\big\|\n f_i(\ol{\x}_1)\big\|^2. 
\end{align}
where recall that~$\Psi$ is the global composite objective function.
We note that~$\mc{O}(\cdot)$ in this section only hides universal constants that are not related to the problem parameters. 

\begin{theorem}[Convergence of \texttt{ProxGT-SA}]\label{thm_SA}
Consider Problem~\eqref{P} with the population risk~\eqref{P_online}.
Let Assumption~\ref{asp_bvr} hold and $\epsilon$ be the error tolerance.
Set $K \asymp \frac{\log(n\zeta)}{1-\ltrue}$, $\a \asymp \frac{1}{L}$, $b \asymp \frac{\nu^2}{n\epsilon^2}$ in Algorithm~\ref{PGT_SA}. Then Algorithm~\ref{PGT_SA} finds an~$\epsilon$-stationary solution in $\mc{O}\big(\frac{L\Delta}{\epsilon^2}\big)$ iterations, leading to $$\mc{O}\left(\frac{L\Delta\nu^2}{n\epsilon^4}\right)$$ stochastic gradient samples at each node and $$\mc{O}\left(\frac{L\Delta}{\epsilon^2}\cdot\frac{\log(n\zeta)}{1-\ltrue}\right)$$ rounds of communication over the network.
\end{theorem}

In view of Theorem~\ref{thm_SA}, \texttt{ProxGT-SA} achieves an \emph{optimal} and \emph{topology-independent} sample complexity at each node that exhibits linear speedup against the centralized minibatch stochastic proximal gradient method \cite{optimal_dsgt,mirror_prox_Lan,lowerbound_sgd} execuated on a single node. To the best of our knowledge, this is the first time such sample complexity result is established, in the literature of decentralized stochastic gradient methods for the general non-convex composite population risk minimization problem.

\begin{theorem}[Convergence of \texttt{ProxGT-SR-O}]\label{thm_SR_O}
Consider Problem~\eqref{P} with the population risk~\eqref{P_online}.
Let Assumption~\ref{asp_bvr} and~\ref{asp_mss} hold and let $\epsilon$ be the error tolerance.
Set $K \asymp \frac{\log(n\zeta)}{1-\ltrue}$, $\a \asymp \frac{1}{L}$, $q \asymp \frac{\nu}{\epsilon}$, $b \asymp \frac{\nu}{n\epsilon}$, $B \asymp \frac{\nu^2}{n\epsilon^2}$ in Algorithm~\ref{PGT_SR_O}. Then Algorithm~\ref{PGT_SR_O} finds an~$\epsilon$-stationary solution in $\mc{O}\big(\frac{L\Delta}{\epsilon^2}+\frac{\nu}{\epsilon}\big)$ iterations, leading to$$\mc{O}\left(\frac{L\Delta\nu}{n\epsilon^3} + \frac{\nu^2}{n\epsilon^2}\right)$$ stochastic gradient samples at each node and $$\mc{O}\left(\left(\frac{L\Delta}{\epsilon^2}+\frac{\nu}{\epsilon}\right)\frac{\log(n\zeta)}{1-\ltrue}\right)$$ rounds of communication over the network.
\end{theorem}
\noindent
Two remarks are in order.
\begin{itemize}
\item Theorem~\ref{thm_SR_O} shows that \texttt{ProxGT-SR-O} attains an \emph{optimal} and \emph{topology-independent} sample complexity at each node that exhibits linear speedup compared to the centralized optimal proximal online variance reduction methods \cite{sarah_ncvx,spiderboost,book_lan,HSARAH,lowerbound_sgd} implemented on a single node. To the best of our knowledge, this appears to be the first such sample complexity result in the literature of the decentralized non-convex composite population risk minimization problem with mean-squared smoothness.

\item For the special case $h = 0$, \texttt{ProxGT-SR-O} also improves the state-of-the-art sample complexity result given by \texttt{GT-HSGD}~\cite{HSGD_Xin} in the following sense. \texttt{GT-HSGD} achieves the optimal sample complexity in the regime where the error tolerance $\epsilon$ of the problem is small enough, i.e., $\epsilon\lesssim(1-\ltrue)^{3}n^{-1}$. \texttt{ProxGT-SR-O} removes this regime restriction by performing $K \asymp \frac{\log(n\zeta)}{1-\ltrue}$ rounds of consensus update per iteration.
\end{itemize}

\begin{theorem}[Convergence of \texttt{ProxGT-SR-E}]\label{thm_SR_E}
Consider Problem~\eqref{P} with the empirical risk~\eqref{P_offline}.
Let Assumption~\ref{asp_bvr} and~\ref{asp_mss} hold and let $\epsilon$ be the error tolerance.
Set $K \asymp \frac{\log\zeta}{1-\ltrue}$, $\a \asymp \frac{1}{L}$, $q \asymp \sqrt{nm}$, $b \asymp \max\left\{\sqrt{\frac{m}{n}},1\right\}$ in Algorithm~\ref{PGT_SR_E}. Then Algorithm~\ref{PGT_SR_E} finds an~$\epsilon$-stationary solution in $\mc{O}\big(\frac{L\Delta}{\epsilon^2}+\sqrt{nm}\big)$ iterations, leading to
$$\mc{O}\left(\frac{L\Delta}{\epsilon^2}\max\left\{\sqrt{\frac{m}{n}},1\right\}+\max\left\{m,\sqrt{nm}\right\}\right)$$ 
stochastic gradient samples at each node and $$\mc{O}\left(\left(\frac{L\Delta}{\epsilon^2}+\sqrt{nm}\right)\frac{\log\zeta}{1-\ltrue}\right)$$ rounds of communication over the network.
\end{theorem}
\noindent
Two remarks are in place:
\begin{itemize}
\item We conclude from Theorem~\ref{thm_SR_E} that under a moderate big-data condition $m\gtrsim n$, \texttt{ProxGT-SR-E} achieves a topology-independent sample complexity of $\mc{O}\left(\frac{L\Delta}{\epsilon^2}\sqrt{\frac{m}{n}}+m\right)$ at each node, leading to a linear speedup compared to the centralized optimal proximal finite-sum variance reduction methods \cite{sarah_ncvx,spiderboost,book_lan} implemented on a single node. To the best of our knowledge, this is the first such sample complexity result for the decentralized non-convex composite empirical risk minimization problem.



\item For the special case $h = 0$, \texttt{ProxGT-SR-E} also improves the state-of-the-art sample complexity result achieved by \texttt{GT-SARAH}~\cite{GT-SARAH} in the following sense. \texttt{GT-SARAH} matches the centralized optimal methods in the regime that the local sample size is large enough, i.e., $m\gtrsim n(1-\ltrue)^{-6}$. \texttt{ProxGT-SR-E} improves this regime to $m\gtrsim n$ by performing $K \asymp \frac{\log(n\zeta)}{1-\ltrue}$ rounds of consensus updates per iteration.
\end{itemize}

\subsection{Improving Communication Complexity via Accelerated Consensus}
As a standard practice, it is possible to employ accelerated consensus algorithms, e.g.,~\cite{acc_ac_morse,acc_ac_olshevsky,best_neurips}, to implement the multiple consensus step $\W^K$ in \texttt{ProxGT} to achieve improved communication complexities. The basic intuition is that the standard consensus algorithm $\x_{t+1} = \W\x_{t}$ returns an $\delta$-accurate average of the initial states $\frac{1}{n}\sum_{i=1}^n\x_1^i$ in $\mc{O}\big(\frac{1}{1-\lambda}\log\frac{1}{\delta}\big)$ rounds of communication, while the accelerated algorithms~\cite{acc_ac_morse,acc_ac_olshevsky,best_neurips}, only take $\mc{O}\big(\frac{1}{\sqrt{1-\lambda}}\log\frac{1}{\delta}\big)$ rounds of communication.

In particular, we can replace $\W^K$ by a Chebyshev type polynomial of $\W$; see, for instance,~\cite[Section 4.2]{best_neurips},~\cite[Section 3.2]{AccDVR}, and~\cite[Section V-C]{PD_Xu} for detailed constructions. In this case, the communication complexity of \texttt{ProxGT-SA} stated in Theorem~\ref{thm_SA} improves to
$$\mc{O}\left(\frac{L\Delta}{\epsilon^2}\cdot \frac{\log(n\zeta)}{\sqrt{1-\ltrue}}\right),$$
while the communication complexity of \texttt{ProxGT-SR-O} stated in Theorem~\ref{thm_SR_O} improves to
$$\mc{O}\left(\left(\frac{L\Delta}{\epsilon^2}+\frac{\nu}{\epsilon}\right)\frac{\log(n\zeta)}{\sqrt{1-\ltrue}}\right),$$
and the communication complexity of \texttt{ProxGT-SR-E} stated in Theorem~\ref{thm_SR_E} improves to
$$\mc{O}\left(\left(\frac{L\Delta}{\epsilon^2}+\sqrt{nm}\right)\frac{\log\zeta}{\sqrt{1-\ltrue}}\right).$$
We omit the detailed calculations here for conciseness.

\section{Conclusions}\label{sec_concl}
In this paper, we focus on decentralized non-convex composite optimization problems over networked nodes, where the network cost is the average of local, smooth, possibly non-convex risk functions plus an extended valued, convex, possibly non-differentiable regularizer. To address this general formulation, we introduce a unified framework, called \texttt{ProxGT}, that is built upon local stochastic gradient estimators and a global gradient tracking technique. We construct several different instantiations of this framework by choosing appropriate local estimators for the corresponding problem classes. In particular, we develop~\texttt{ProxGT-SA} and~\texttt{ProxGT-SR-O} for the population risk, and~\texttt{ProxGT-SR-E} for the empirical risk. Remarkably, we show that each algorithm achieves a network topology-independent sample complexity at each node, leading to a linear speedup compared to its centralized optimal counterpart.


{\bibliographystyle{plain}
\bibliography{reference_xin.bib}}

\newpage

\appendix
\section{Proofs of the Main Results}\label{sec_analysis}
In this section, we describe a unified analysis for the proposed \texttt{Prox-GT} framework. Throughout the rest of the paper, we let Assumption~\ref{asp_f},~\ref{asp_net}, and~\ref{asp_unbias} hold without explicit statements. 

\subsection{Preliminaries}\label{app_prelim}
We start by introducing some additional notations for Algorithm~\ref{PGT},~\ref{PGT_SA},~\ref{PGT_SR_O}, and~\ref{PGT_SR_E}. 
We find it convenient to abstract the local proximal descent step by a \emph{stochastic gradient mapping}:~$\forall t\geq1$ and~$i\in\mc{V}$,
\begin{align}\label{sgp}
\sgp_{t}^i 
:= \frac{1}{\alpha}\big(\x_t^i - \z_{t+1}^{i}\big).
\end{align}
For all~$t\geq1$, we let
\begin{align*}
\sgp_t := 
\begin{bmatrix}
\sgp_t^1\\
\cdots \\
\sgp_t^n
\end{bmatrix},
\qquad
\nf(\x_t) := 
\begin{bmatrix}
\n f_1(\x_t^1)\\
\cdots \\
\n f_n(\x_t^n)
\end{bmatrix},
\end{align*}
and define the following network mean states:
\begin{align*}
&\ol{\x}_t := \frac{1}{n}\sum_{i=1}^n \x_t^i,
\qquad
\ol{\y}_t := \frac{1}{n}\sum_{i=1}^n \y_t^i,
\qquad
\ol{\z}_t := \frac{1}{n}\sum_{i=1}^n \z_t^i,
\nonumber\\
&
\ol{\mb{v}}_t := \frac{1}{n}\sum_{i=1}^n \mb{v}_t^i,   
\qquad
\ol{\sgp}_t := \frac{1}{n}\sum_{i=1}^n\sgp_t^i,
\qquad
\ol{\nf}(\x_t) := \frac{1}{n}\sum_{i=1}^n\n f_i(\x_t^i).
\end{align*}
In addition, we define the exact averaging matrix $$\J := \left(\frac{1}{n}\mb{1}_n\mb{1}_n^\top\right) \otimes \I_p.$$
Averaging~\eqref{sgp} over~$i$ from~$1$ to~$n$ gives:~$\forall t\geq1$,
\begin{align}\label{sgp_ave}
\ol{\z}_{t+1} 
= \ol{\x}_{t} - \a\ol{\sgp}_{t}.
\end{align}
We multiply~$\frac{1}{n}(\mb{1}_n^\top\otimes\I_p)$ to the~$\x$-update of Algorithm~\ref{PGT} to obtain:~$\forall t\geq1$,
\begin{align}\label{z_track}
\ol{\x}_{t+1} = \ol{\z}_{t+1}.
\end{align}
Combining~\eqref{sgp_ave} and~\eqref{z_track} yields:~$\forall t\geq1$,
\begin{align}\label{sgp_ave_d}
\ol{\x}_{t+1}    
= \ol{\x}_{t} - \a\ol{\sgp}_{t}.
\end{align}
Throughout the analysis, we fix arbitrary~$K\geq1$ and denote
\begin{align}\label{lambda_K}
\lambda := \ltrue^K.   
\end{align}

\subsection{Basic Facts}
This section presents several basic facts that are used frequently in our analysis.
We make use of a well-known non-expansiveness result for proximal mappings. 
\begin{lem}[{\cite[Theorem 6.4.2]{book_Beck2}}]\label{lem_prox_nonexp}
Let~$h:\R^{p}\ra\R\cup\{+\infty\}$ be a proper, closed, and convex function. Then we have the following:
$$
\|\p_{h}(\x) - \p_{h}(\y)\|
\leq \|\x - \y\|, 
\quad
\forall \x,\y\in\R^p.
$$
\end{lem}
For ease of reference, we give a trivial accumulation formula for scalar sequences with contraction. 
\begin{lem}\label{lem_acc}
Let~$\{a_t\}$ and $\{b_t\}$ be scalar sequences and~$0<q<1$, such that $$a_{t+1} \leq q a_{t} + b_t, \qquad\forall t\geq1.$$ Then for all $T\geq2$, we have $$\sum_{t=1}^{T}a_t\leq \frac{1}{1-q}a_1 + \frac{1}{1-q}\sum_{t=1}^{T-1}b_t$$ 
and 
$$\sum_{t=2}^{T+1}a_{t} 
\leq\frac{1}{1-q}a_2  
+ \frac{1}{1-q}\sum_{t=2}^{T}b_t.$$
\end{lem}
\begin{proof}
The proof follows from standard arguments of convolution sums and we omit the details.
\end{proof}

The following lemma concerns the interplay between~$\W$ and~$\J$. We provide its proof for completeness.

\begin{lem}\label{lem_weight}
The following statements hold for all~$K\geq1$.
\begin{enumerate}[(a)]
\item $\W^K\J = \J\W^K = \J$. \label{W_equal}
\item $\big\|\W^K - \J\big\| = \ltrue^K$. \label{W_spectral}
\item $\big\|\W^K\x - \J\x\big\| \leq \ltrue^K\big\|\x - \J\x\big\|, \forall \x\in\R^{np}$. \label{W_contract}
\end{enumerate}
\end{lem}
\begin{proof}
Since~$\Wtrue$ is doubly stochastic, we have
\begin{align}\label{commute}
\W\J = \J\W = \J,   
\end{align}
which leads to part~\eqref{W_equal} by induction. Part~\eqref{W_spectral} follows from
\begin{align*}
\big\|\W^K - \J\big\| 
= \big\|(\W - \J)^K\big\|
= \ltrue^K,
\end{align*}
where the first equality uses~\eqref{commute} and the second equality uses the definition of the spectral norm of a matrix. Finally, part~\eqref{W_contract} is due to
\begin{align*}
\big\|\W^K\x - \J\x\big\|
= \big\|(\W^K - \J)(\x-\J\x)\big\|
\leq \big\|\W^K - \J\big\| \big\|\x - \J\x\big\|
= \ltrue^K\big\|\x - \J\x\big\|,
\end{align*}
where the first equality uses part~\eqref{W_equal} and~$\J^2 = \J$, and the last equality uses part~\eqref{W_spectral}.
\end{proof}

Finally, we present a simple yet useful decomposition inequality.

\begin{lem}\label{lem_x_diff}
Consider the iterates generated by Algorithm~\ref{PGT}. Then we have: $\forall T\geq2$,
\begin{align*}
\sum_{t=2}^{T}\|\x_t - \x_{t-1}\|^2
\leq 6\sum_{t=1}^{T}\|\x_t - \J\x_t\|^2
+ 3n\a^2\sum_{t=1}^{T-1}\|\ol{\sgp}_{t}\|^2,
\end{align*}
\end{lem}
\begin{proof}
We note that~$\forall t\geq2$,
\begin{align}\label{x_diff_0}
\|\x_t - \x_{t-1}\|^2
=&~\big\|\x_t - \J\x_t + \J\x_t - \J\x_{t-1} + \J\x_{t-1} - \x_{t-1}\big\|^2
\nonumber\\
\leq&~3\|\x_t - \J\x_t\|^2 + 3n\|\ol{\x}_t - \ol{\x}_{t-1}\|^2 + 3\|\x_{t-1} - \J\x_{t-1}\|^2,
\nonumber\\
\leq&~3\|\x_t - \J\x_t\|^2 + 3n\a^2\|\ol{\sgp}_{t-1}\|^2 + 3\|\x_{t-1} - \J\x_{t-1}\|^2,
\end{align}
where the second line uses~\eqref{sgp_ave_d}. Summing up~\eqref{x_diff_0} gives
\begin{align*}
\sum_{t=2}^{T}\|\x_t - \x_{t-1}\|^2
\leq&~3\sum_{t=2}^{T}\Big(\|\x_t - \J\x_t\|^2 + \|\x_{t-1} - \J\x_{t-1}\|^2\Big)
+ 3n\a^2\sum_{t=2}^{T}\|\ol{\sgp}_{t-1}\|^2
\nonumber\\
\leq&~6\sum_{t=1}^{T}\|\x_t - \J\x_t\|^2 + 
+ 3n\a^2\sum_{t=2}^{T}\|\ol{\sgp}_{t-1}\|^2
\end{align*}
which finishes the proof.
\end{proof}

\subsection{Descent Inequality and Error Bounds}
We first establish a key descent inequality in terms of the value of the global composite objective function~$\Psi$. This result plays a central role in our analysis.
\begin{lem}[\textbf{Descent}]\label{lem_descent}
Consider the iterates generated by Algorithm~\ref{PGT}. If~$0<\alpha\leq \frac{1}{8L}$, then we have:~$\forall t\geq1$,
\begin{align*}
\frac{1}{n}\sum_{t=1}^T\left(\sum_{i=1}^n\left\|\gp(\x_t^i)\right\|^2
+ L^2\left\|\x_{t} - \J\x_{t}\right\|^2\right) 
\leq&~\frac{8\Delta}{\a}
- \frac{1}{n}\sum_{t=1}^T\sum_{i=1}^n\left\|\sgp_{t}^i\right\|^2
+ 76\sum_{t=1}^T\left\|\ol{\mb{v}}_{t} - \ol{\nf}(\x_t)\right\|^2
\nonumber\\
&+ \frac{6}{\alpha^2n}\sum_{t=1}^T\left\|\x_{t} - \J\x_{t}\right\|^2
+ \frac{10}{n}\sum_{t=2}^{T+1}\|\y_{t} - \J\y_{t}\|^2.
\end{align*}
\end{lem}
\begin{proof}
See Appendix~\ref{app_proof_descent_lem}.
\end{proof}
In light of Lemma~\ref{lem_descent}, our analysis approach is to show that the accumulated descent effect of the stochastic gradient mappings $\sum_{i=1}^{n}\|\sgp_t^i\|$ dominates the accumulated consensus, variance, and gradient tracking errors up to constant factors. To this aim, we establish useful error bounds for different algorithms. The following one is a consequence of the non-expansiveness of the proximal operator.

\begin{lem}[\textbf{Consensus}]\label{lem_consensus}
Consider the iterates generated by Algorithm~\ref{PGT}. We have:~$\forall t\geq1$,
\begin{align*}
\sum_{t=1}^T\big\|\x_{t}-\J\x_{t}\big\|^2
\leq \frac{4\lambda^2\alpha^2}{(1-\lambda^2)^2}\sum_{t=2}^{T}\big\|\y_{t} - \J\y_{t}\big\|^2.
\end{align*}
\end{lem}
\begin{proof}
See Appendix~\ref{app_proof_consensus_lem}.
\end{proof}
\begin{rmk}
It is worth noting that Lemma~\ref{lem_descent} and~\ref{lem_consensus} do not use any properties of the gradient estimator~$\mb{v}_t$. 
Therefore they may be of independent interest and used in other decentralized stochastic proximal gradient type methods for non-convex composite problems.
\end{rmk}
The next lemma establishes variance bounds for different algorithms.
\begin{lem}[\textbf{Variance}]\label{lem_VR}
The following statements hold.
\begin{enumerate}[(a)]
\item Let Assumption~\ref{asp_bvr} hold and consider the iterates generated by Algorithm~\ref{PGT_SA}. Then we have:~$\forall t\geq1$,
\begin{align*}
\E\Big[\big\|\ol{\mb{v}}_t - \ol{\nf}(\x_t)\big\|^2\Big] \leq \frac{\nu^2}{nb}.  
\end{align*}
\label{lem_VR_M}
\item Let Assumption~\ref{asp_bvr} and~\ref{asp_mss} hold. Consider the iterates generated by Algorithm~\ref{PGT_SR_O}.
Suppose that $T = Rq$ for some $R\in\mathbb{Z}^{+}$.
Then we have:~$\forall T\geq q$,
\begin{align*}
\sum_{t=1}^T\E\Big[\big\|\ol{\mb{v}}_t - \ol{\nf}(\x_t)\big\|^2\Big]
\leq
\frac{6L^2q}{n^2b}\sum_{t=1}^T\E\Big[\big\|\mb{x}_t - \J\mb{x}_{t}\big\|^2\Big]
+ \frac{qL^2\a^2}{nb}\sum_{t=1}^{T-1}\E\Big[\big\|\ol{\sgp}_{t}\big\|^2\Big]
+ \frac{T\nu^2}{nB}.
\end{align*}
\label{lem_VR_O}
\item Let Assumption~\ref{asp_mss} hold. Consider the iterates generated by Algorithm~\ref{PGT_SR_E}. 
Suppose that $T = Rq$ for some $R\in\mathbb{Z}^{+}$.
Then we have:~$\forall T\geq q$,
\begin{align*}
\sum_{t=1}^T\E\Big[\big\|\ol{\mb{v}}_t - \ol{\nf}(\x_t)\big\|^2\Big]
\leq
\frac{6L^2q}{n^2b}\sum_{t=1}^T\E\Big[\big\|\mb{x}_t - \J\mb{x}_{t}\big\|^2\Big]
+ \frac{qL^2\a^2}{nb}\sum_{t=1}^{T-1}\E\Big[\big\|\ol{\sgp}_{t}\big\|^2\Big].
\end{align*}
\label{lem_VR_E}
\end{enumerate}
\end{lem}
\begin{proof}
See Appendix~\ref{app_proof_VR_lem}.
\end{proof}

Finally, we give tracking error bounds for different algorithms in the following lemma. 
\begin{lem}[\textbf{Tracking}]\label{lem_GT}
The following statements hold.
\begin{enumerate}[(a)]
\item Let Assumption~\ref{asp_bvr} hold and consider the iterates generated by Algorithm~\ref{PGT_SA}. Then we have:~$\forall T\geq2$,
\begin{align*}
\sum_{t=2}^{T+1}\E\big[\|\y_{t} -\mb{J}\y_{t}\|^2\big]
\leq&~\frac{2\lambda^2n\zeta^2}{1-\lambda^2}
+ \frac{12\lambda^2n\a^2 L^2}{(1-\lambda^2)^2}\sum_{t=1}^{T-1}\E\big[\|\ol{\sgp}_t\|^2\big] 
+\frac{24\lambda^2 L^2}{(1-\lambda^2)^2}\sum_{t=1}^T\E\big[\|\x_t - \J\x_{t}\|^2\big] \nonumber\\
&+ \frac{4T(2\lambda^2n+1)\nu^2}{b(1-\lambda^2)}.
\end{align*}
\label{lem_GT_M}
\item Let Assumption~\ref{asp_bvr} and~\ref{asp_mss} hold. Consider the iterates generated by Algorithm~\ref{PGT_SR_O}. Let $T = Rq$ for some $R\in\mathbb{Z}^{+}$ and $R\geq 2$. Then we have:
\begin{align*}
\sum_{t=2}^{T+1}\E\big[\|\y_{t} -\mb{J}\y_{t}\|^2\big]
\leq&~\frac{2\lambda^2n\zeta^2}{1-\lambda^2}
+ \frac{96\lambda^2L^2}{(1-\lambda^2)^2}\sum_{t=1}^{T}\E\big[\|\x_t - \J\x_{t}\|^2\big]
+ \frac{48\lambda^2n\a^2L^2}{(1-\lambda^2)^2}\sum_{t=1}^{T-1}\E\big[\|\ol{\sgp}_{t}\|^2\big]
\nonumber\\
&+ \frac{14\lambda^2Tn\nu^2}{(1-\lambda^2)^2Bq}.
\end{align*}
\label{lem_GT_O}
\item Let Assumption~\ref{asp_mss} hold. Consider the iterates generated by Algorithm~\ref{PGT_SR_E}. Let $T = Rq$ for some $R\in\mathbb{Z}^{+}$ and $R\geq 2$. Then we have:
\begin{align*}
\sum_{t=2}^{T+1}\E\big[\|\y_{t} -\mb{J}\y_{t}\|^2\big] 
\leq&~\frac{2\lambda^2n\zeta^2}{1-\lambda^2}
+ \frac{96\lambda^2L^2}{(1-\lambda^2)^2}\sum_{t=1}^{T}\E\big[\|\x_t - \J\x_{t}\|^2\big]
+ \frac{48\lambda^2n\a^2L^2}{(1-\lambda^2)^2}\sum_{t=1}^{T-1}\E\big[\|\ol{\sgp}_{t}\|^2\big].
\end{align*}
\label{lem_GT_E}
\end{enumerate}
\end{lem}
\begin{proof}
See Appendix~\ref{app_proof_GT_lem}.
\end{proof}

\subsection{Proofs of the Main Theorems}
We first use the consensus error bound in Lemma~\ref{lem_consensus} to refine the descent inequality in Lemma~\ref{lem_descent}.
\begin{prop}\label{PGT_main0}
Consider the iterates generated by Algorithm~\ref{PGT}. If~$0<\alpha\leq \frac{1}{8L}$, then we have:~$\forall t\geq1$, 
\begin{align*}
\frac{1}{n}\sum_{t=1}^T\left(\sum_{i=1}^n\left\|\gp(\x_t^i)\right\|^2
+ L^2\left\|\x_{t} - \J\x_{t}\right\|^2\right) 
\leq&~\frac{8\Delta}{\a}
- \sum_{t=1}^T\left\|\ol{\sgp}_{t}\right\|^2
+ 76\sum_{t=1}^T\left\|\ol{\mb{v}}_{t} - \ol{\nf}(\x_t)\right\|^2
\nonumber\\
&+ \frac{34}{(1-\lambda^2)^2n}\sum_{t=2}^{T+1}\|\y_{t} - \J\y_{t}\|^2.
\end{align*}
\end{prop}
\begin{proof}
This result follows by applying Lemma~\ref{lem_consensus} to Lemma~\ref{lem_descent} and $\|\ol{\sgp}_t\|^2 \leq \frac{1}{n}\sum_{i=1}^{n}\|\sgp_t^i\|^2$. 
\end{proof}

\subsubsection{Proof of Theorem \ref{thm_SA}}
We apply Lemma~\ref{lem_consensus} to Lemma~\ref{lem_GT}\eqref{lem_GT_M} to obtain: $\forall T\geq2$,
\begin{align}\label{SA_main_0}
\left(1-\frac{96\lambda^4\alpha^2 L^2}{(1-\lambda^2)^4}\right)\sum_{t=2}^{T+1}\E\big[\|\y_{t} -\mb{J}\y_{t}\|^2\big] 
\leq&~\frac{2\lambda^2n\zeta}{1-\lambda^2}
+ \frac{12\lambda^2n\a^2 L^2}{(1-\lambda^2)^2}\sum_{t=1}^{T-1}\E\big[\|\ol{\sgp}_t\|^2\big]
+ \frac{4T(2\lambda^2n+1)\nu^2}{b(1-\lambda^2)}.
\end{align}
If~$0<\a\leq\frac{(1-\lambda^2)^2}{14\lambda^2 L}$, then
$1-\frac{96\lambda^4\alpha^2 L^2}{(1-\lambda^2)^4}\geq\frac{1}{2}$ and hence~\eqref{SA_main_0} implies that~$\forall T\geq2$,
\begin{align}\label{SA_main_1}
\sum_{t=2}^{T+1}\E\big[\|\y_{t} -\mb{J}\y_{t}\|^2\big] 
\leq&~\frac{4\lambda^2n\zeta^2}{1-\lambda^2}
+ \frac{24\lambda^2n\a^2 L^2}{(1-\lambda^2)^2}\sum_{t=1}^{T-1}\E\big[\|\ol{\sgp}_t\|^2\big]
+ \frac{8T(2\lambda^2n+1)\nu^2}{b(1-\lambda^2)}.
\end{align}
Plugging Lemma~\ref{lem_VR}\eqref{lem_VR_M} and~\eqref{SA_main_1} into Proposition~\ref{PGT_main0} gives: if~$0<\a\leq\min\Big\{\frac{(1-\lambda^2)^2}{14\lambda^2 },\frac{1}{8}\Big\}\frac{1}{L}$, then
\begin{align}\label{SA_main2}
&\frac{1}{n}\sum_{t=1}^T\sum_{i=1}^n\E\Big[\big\|\gp(\x_t^i)\big\|^2
+ L^2\big\|\x_{t}^i - \ol{\x}_{t}\big\|^2\Big] 
\nonumber\\
\leq&~\frac{8\Delta}{\a}
- \sum_{t=1}^T\E\big[\|\ol{\sgp}_{t}\|^2\big]
+ \frac{76T\nu^2}{nb} 
+ \frac{272T(2\lambda^2n+1)\nu^2}{nb(1-\lambda^2)^3}
+ \frac{136\lambda^2\zeta^2}{(1-\lambda^2)^3} 
+ \frac{816\lambda^2\a^2 L^2}{(1-\lambda^2)^4}\sum_{t=1}^{T-1}\E\big[\|\ol{\sgp}_t\|^2\big]
\nonumber\\
\leq&~\frac{8\Delta}{\a}
- \left(1 - \frac{816\lambda^2\a^2 L^2}{(1-\lambda^2)^4}\right)\sum_{t=1}^T\E\big[\|\ol{\sgp}_{t}\|^2\big]
+ \frac{136\lambda^2\zeta^2}{(1-\lambda^2)^3} 
+ \frac{348T\nu^2}{nb(1-\lambda^2)^3} 
+ \frac{544\lambda^2 T\nu^2}{b(1-\lambda^2)^3}.
\end{align}
From~\eqref{SA_main2}, we have: if~$0<\a\leq\min\Big\{\frac{(1-\lambda^2)^2}{30\lambda },\frac{1}{8}\Big\}\frac{1}{L}$, then~$\forall T\geq2$,
\begin{align}\label{SA_main3}
\frac{1}{nT}\sum_{t=1}^T\sum_{i=1}^n\E\Big[\big\|\gp(\x_t^i)\big\|^2
+ L^2\big\|\x_{t}^i - \ol{\x}_{t}\big\|^2\Big]
\leq \frac{8\Delta}{\a T}
+ \frac{136\lambda^2\zeta^2}{(1-\lambda^2)^3 T} 
+ \frac{348\nu^2}{nb(1-\lambda^2)^3} 
+ \frac{544\lambda^2 \nu^2}{b(1-\lambda^2)^3}.
\end{align}
Recall from~\eqref{lambda_K} that~$\lambda := \ltrue^K$ and we set 
\begin{align*}
K \asymp \frac{\log(n\zeta)}{1-\ltrue},
\end{align*}
so that
$\frac{1}{1-\lambda} = \mc{O}(1),
\lambda\zeta = \mc{O}(1),
\lambda n = \mc{O}(1).$
As a consequence, from~\eqref{SA_main3} we have: if~$0<\a\lesssim\frac{1}{L}$, then
\begin{align}\label{SA_main4}
\frac{1}{nT}\sum_{t=1}^T\sum_{i=1}^n\E\Big[\big\|\gp(\x_t^i)\big\|^2
+ L^2\big\|\x_{t}^i - \ol{\x}_{t}\big\|^2\Big]
\lesssim \frac{\Delta}{\a T}
+ \frac{\nu^2}{nb}.
\end{align}
Finally, we observe that choosing
\begin{align*}
\a \asymp \frac{1}{L},
\qquad
b \asymp \frac{\nu^2}{n\epsilon^2},
\qquad
T \asymp \frac{L\Delta}{\epsilon^{2}}
\end{align*}
in~\eqref{SA_main4} gives $\frac{1}{nT}\sum_{t=1}^T\sum_{i=1}^n\E\big[\|\gp(\x_t^i)\|^2
+ L^2\|\x_{t}^i - \ol{\x}_{t}\|^2\big]
\lesssim \epsilon^2.$ The ensuing complexity results follow from the fact that each iteration of Algorithm~\ref{PGT_SA} incurs $b$ stochastic gradient samples and $K$ rounds of communication.

\subsubsection{Proof of Theorem~\ref{thm_SR_O}}
Consider $T = Rq$ for some $R\in\mathbb{Z}^{+}$ and $R\geq2$.
Plugging Lemma~\ref{lem_consensus} to Lemma~\ref{lem_VR}\eqref{lem_VR_O} gives: 
\begin{align*}
\sum_{t=1}^T\E\Big[\big\|\ol{\mb{v}}_t - \ol{\nf}(\x_t)\big\|^2\Big]
\leq
\frac{24\lambda^2\alpha^2L^2q}{(1-\lambda^2)^2n^2b}\sum_{t=2}^{T}\E\Big[\big\|\y_{t} - \J\y_{t}\big\|^2\Big]
+ \frac{qL^2\a^2}{nb}\sum_{t=1}^{T-1}\E\Big[\big\|\ol{\sgp}_{t}\big\|^2\Big]
+ \frac{T\nu^2}{nB}.
\end{align*}
In particular, if $0<\a\leq\sqrt{\frac{nb}{24q}}\frac{1}{L}$, we have: 
\begin{align}\label{SRO_main_0}
\sum_{t=1}^T\E\Big[\big\|\ol{\mb{v}}_t - \ol{\nf}(\x_t)\big\|^2\Big]
\leq
\frac{\lambda^2}{(1-\lambda^2)^2n}\sum_{t=2}^{T}\E\Big[\big\|\y_{t} - \J\y_{t}\big\|^2\Big]
+ \frac{qL^2\a^2}{nb}\sum_{t=1}^{T-1}\E\Big[\big\|\ol{\sgp}_{t}\big\|^2\Big]
+ \frac{T\nu^2}{nB}.
\end{align}
Applying~\eqref{SRO_main_0} to Proposition~\ref{PGT_main0} yields: if $0<\a\leq\min\left\{\frac{1}{8},\sqrt{\frac{nb}{24q}}\right\}\frac{1}{L}$, then 
\begin{align*}
\frac{1}{n}\sum_{t=1}^T\sum_{i=1}^n\E\Big[\big\|\gp(\x_t^i)\big\|^2
+ L^2\big\|\x_{t}^i - \ol{\x}_{t}\big\|^2\Big] 
\leq&~\frac{8\Delta}{\a}
- \left(1-\frac{76qL^2\a^2}{nb}\right)\sum_{t=1}^T\E\Big[\big\|\ol{\sgp}_{t}\big\|^2\Big]
\nonumber\\
&+ \frac{110}{(1-\lambda^2)^2n}\sum_{t=2}^{T+1}\E\Big[\big\|\y_{t} - \J\y_{t}\big\|^2\Big]
+ \frac{76T\nu^2}{nB}.
\end{align*}
In particular, if $0<\a\leq\min\left\{\frac{1}{8},\sqrt{\frac{nb}{152q}}\right\}\frac{1}{L}$, we have: 
\begin{align}\label{SRO_main_1}
\frac{1}{n}\sum_{t=1}^T\sum_{i=1}^n\E\Big[\big\|\gp(\x_t^i)\big\|^2
+ L^2\big\|\x_{t}^i - \ol{\x}_{t}\big\|^2\Big] 
\leq&~\frac{8\Delta}{\a}
- \frac{1}{2}\sum_{t=1}^T\E\Big[\big\|\ol{\sgp}_{t}\big\|^2\Big]
+ \frac{110}{(1-\lambda^2)^2n}\sum_{t=2}^{T+1}\E\Big[\big\|\y_{t} - \J\y_{t}\big\|^2\Big]
+ \frac{76T\nu^2}{nB}.
\end{align}
To proceed, we apply Lemma~\ref{lem_consensus} to Lemma~\ref{lem_GT}\eqref{lem_GT_O} to obtain: 
\begin{align}\label{SRO_main_2}
\left(1-\frac{384\lambda^4\a^2L^2}{(1-\lambda^2)^4}\right)\sum_{t=2}^{T+1}\E\big[\|\y_{t} -\mb{J}\y_{t}\|^2\big]
\leq&~\frac{2\lambda^2n\zeta^2}{1-\lambda^2}
+ \frac{48\lambda^2n\a^2L^2}{(1-\lambda^2)^2}\sum_{t=1}^{T-1}\E\big[\|\ol{\sgp}_{t}\|^2\big]
+ \frac{14\lambda^2Tn\nu^2}{(1-\lambda^2)^2Bq}.
\end{align}
If~$0<\a\leq\frac{(1-\lambda^2)^2}{28\lambda^2 L}$, \eqref{SRO_main_2} implies that
\begin{align}\label{SRO_main_3}
\sum_{t=2}^{T+1}\E\big[\|\y_{t} -\mb{J}\y_{t}\|^2\big]
\leq&~\frac{4\lambda^2n\zeta^2}{1-\lambda^2}
+ \frac{96\lambda^2n\a^2L^2}{(1-\lambda^2)^2}\sum_{t=1}^{T-1}\E\big[\|\ol{\sgp}_{t}\|^2\big]
+ \frac{28\lambda^2Tn\nu^2}{(1-\lambda^2)^2Bq}.
\end{align}
Finally, plugging~\eqref{SRO_main_3} to~\eqref{SRO_main_1}, we obtain: if $0<\a\leq\min\left\{\frac{1}{8},\sqrt{\frac{nb}{152q}},\frac{(1-\lambda^2)^2}{28\lambda^2}\right\}\frac{1}{L}$, then
\begin{align*}
\frac{1}{n}\sum_{t=1}^T\sum_{i=1}^n\E\Big[\big\|\gp(\x_t^i)\big\|^2
+ L^2\big\|\x_{t}^i - \ol{\x}_{t}\big\|^2\Big] 
\leq&~\frac{8\Delta}{\a}
+ \frac{76T\nu^2}{nB} 
+ \frac{440\lambda^2\zeta^2}{(1-\lambda^2)^3}
+ \frac{3080\lambda^2T\nu^2}{(1-\lambda^2)^4Bq}
\nonumber\\
&- \frac{1}{2}\left(1-\frac{21120\lambda^2\a^2L^2}{(1-\lambda^2)^4}\right)\sum_{t=1}^T\E\Big[\big\|\ol{\sgp}_{t}\big\|^2\Big]
.
\end{align*}
Hence, if $0<\a\leq\min\left\{\frac{1}{8},\sqrt{\frac{nb}{152q}},\frac{(1-\lambda^2)^2}{146\lambda^2}\right\}\frac{1}{L}$, then
\begin{align}\label{SRO_main}
\frac{1}{nT}\sum_{t=1}^T\sum_{i=1}^n\E\Big[\big\|\gp(\x_t^i)\big\|^2
+ L^2\big\|\x_{t}^i - \ol{\x}_{t}\big\|^2\Big] 
\leq&~\frac{8\Delta}{\a T}
+ \frac{76\nu^2}{nB} 
+ \frac{440\lambda^2\zeta^2}{(1-\lambda^2)^3T}
+ \frac{3080\lambda^2\nu^2}{(1-\lambda^2)^4Bq}.
\end{align}
Let $\epsilon>0$ be given. Recall from~\eqref{lambda_K} that~$\lambda := \ltrue^K$ and we set 
\begin{align*}
K \asymp \frac{\log(n\zeta)}{1-\ltrue},
\end{align*}
so that
$\frac{1}{1-\lambda} = \mc{O}(1),
\lambda\zeta = \mc{O}(1),
\lambda n = \mc{O}(1);$
moreover, we let
\begin{align*}
q = nb \qquad\text{and}\qquad\a\asymp\frac{1}{L}.
\end{align*}
As a consequence, we have from~\eqref{SRO_main} that
\begin{align}\label{SRO_main_F}
\frac{1}{nT}\sum_{t=1}^T\sum_{i=1}^n\E\Big[\big\|\gp(\x_t^i)\big\|^2
+ L^2\big\|\x_{t}^i - \ol{\x}_{t}\big\|^2\Big] 
\lesssim \frac{L\Delta}{T}
+ \frac{\nu^2}{nB}.
\end{align}
In view of~\eqref{SRO_main_F}, we further choose 
\begin{align}\label{SRO_TB}
T \asymp \frac{L\Delta}{\epsilon^2} + q
\qquad
\text{and}
\qquad
B \asymp \frac{\nu^2}{n\epsilon^2},
\end{align}
which lead to $\frac{1}{nT}\sum_{t=1}^T\sum_{i=1}^n\E\big[\|\gp(\x_t^i)\|^2
+ L^2\|\x_{t}^i - \ol{\x}_{t}\|^2\big]
\lesssim \epsilon^2.$ Since \texttt{ProxGT-SR-O} requires $B$ samples every $q$ iterations and $b$ samples at each iteration, its total sample complexity is bounded by 
\begin{align}\label{SRO_IFO}
\mc{O}\left(T\left(b + \frac{B}{q}\right)\right).
\end{align}
Setting $b = B/q$, together with $q = nb$ stated above, gives
\begin{align}\label{SRO_bq}
b \asymp \sqrt{\frac{B}{n}} = \frac{\nu}{n\epsilon},
\qquad
\text{and}
\qquad
q \asymp \frac{\nu}{\epsilon}.
\end{align}
Applying~\eqref{SRO_TB} and~\eqref{SRO_bq} to~\eqref{SRO_IFO} concludes the ensuing sample complexity and the corresponding communication complexity is given by $TK$.

\subsubsection{Proof of Theorem~\ref{thm_SR_E}}
Consider $T = Rq$ for some $R\in\mathbb{Z}^{+}$ and $R\geq2$.
Plugging Lemma~\ref{lem_consensus} to Lemma~\ref{lem_VR}\eqref{lem_VR_E} gives: 
\begin{align*}
\sum_{t=1}^T\E\Big[\big\|\ol{\mb{v}}_t - \ol{\nf}(\x_t)\big\|^2\Big]
\leq
\frac{24\lambda^2\alpha^2L^2q}{(1-\lambda^2)^2n^2b}\sum_{t=2}^{T}\E\Big[\big\|\y_{t} - \J\y_{t}\big\|^2\Big]
+ \frac{qL^2\a^2}{nb}\sum_{t=1}^{T-1}\E\Big[\big\|\ol{\sgp}_{t}\big\|^2\Big].
\end{align*}
In particular, if $0<\a\leq\sqrt{\frac{nb}{24q}}\frac{1}{L}$, we have: 
\begin{align}\label{SRE_main_0}
\sum_{t=1}^T\E\Big[\big\|\ol{\mb{v}}_t - \ol{\nf}(\x_t)\big\|^2\Big]
\leq
\frac{\lambda^2}{(1-\lambda^2)^2n}\sum_{t=2}^{T}\E\Big[\big\|\y_{t} - \J\y_{t}\big\|^2\Big]
+ \frac{qL^2\a^2}{nb}\sum_{t=1}^{T-1}\E\Big[\big\|\ol{\sgp}_{t}\big\|^2\Big].
\end{align}
Applying~\eqref{SRE_main_0} to Proposition~\ref{PGT_main0} yields: if $0<\a\leq\min\left\{\frac{1}{8},\sqrt{\frac{nb}{24q}}\right\}\frac{1}{L}$, then 
\begin{align*}
\frac{1}{n}\sum_{t=1}^T\sum_{i=1}^n\E\Big[\big\|\gp(\x_t^i)\big\|^2
+ L^2\big\|\x_{t}^i - \ol{\x}_{t}\big\|^2\Big] 
\leq&~\frac{8\Delta}{\a}
- \left(1-\frac{76qL^2\a^2}{nb}\right)\sum_{t=1}^T\E\Big[\big\|\ol{\sgp}_{t}\big\|^2\Big]
\nonumber\\
&+ \frac{110}{(1-\lambda^2)^2n}\sum_{t=2}^{T+1}\E\Big[\big\|\y_{t} - \J\y_{t}\big\|^2\Big].
\end{align*}
In particular, if $0<\a\leq\min\left\{\frac{1}{8},\sqrt{\frac{nb}{152q}}\right\}\frac{1}{L}$, we have: 
\begin{align}\label{SRE_main_1}
\frac{1}{n}\sum_{t=1}^T\sum_{i=1}^n\E\Big[\big\|\gp(\x_t^i)\big\|^2
+ L^2\big\|\x_{t}^i - \ol{\x}_{t}\big\|^2\Big] 
\leq&~\frac{8\Delta}{\a}
- \frac{1}{2}\sum_{t=1}^T\E\Big[\big\|\ol{\sgp}_{t}\big\|^2\Big]
+ \frac{110}{(1-\lambda^2)^2n}\sum_{t=2}^{T+1}\E\Big[\big\|\y_{t} - \J\y_{t}\big\|^2\Big].
\end{align}
To proceed, we apply Lemma~\ref{lem_consensus} to Lemma~\ref{lem_GT}\eqref{lem_GT_O} to obtain: 
\begin{align}\label{SRE_main_2}
\left(1-\frac{384\lambda^4\a^2L^2}{(1-\lambda^2)^4}\right)\sum_{t=2}^{T+1}\E\big[\|\y_{t} -\mb{J}\y_{t}\|^2\big]
\leq&~\frac{2\lambda^2n\zeta^2}{1-\lambda^2}
+ \frac{48\lambda^2n\a^2L^2}{(1-\lambda^2)^2}\sum_{t=1}^{T-1}\E\big[\|\ol{\sgp}_{t}\|^2\big].
\end{align}
If~$0<\a\leq\frac{(1-\lambda^2)^2}{28\lambda^2 L}$, \eqref{SRE_main_2} implies that
\begin{align}\label{SRE_main_3}
\sum_{t=2}^{T+1}\E\big[\|\y_{t} -\mb{J}\y_{t}\|^2\big]
\leq&~\frac{4\lambda^2n\zeta^2}{1-\lambda^2}
+ \frac{96\lambda^2n\a^2L^2}{(1-\lambda^2)^2}\sum_{t=1}^{T-1}\E\big[\|\ol{\sgp}_{t}\|^2\big].
\end{align}
Finally, plugging~\eqref{SRE_main_3} to~\eqref{SRE_main_1}, we obtain: if $0<\a\leq\min\left\{\frac{1}{8},\sqrt{\frac{nb}{152q}},\frac{(1-\lambda^2)^2}{28\lambda^2}\right\}\frac{1}{L}$, then
\begin{align*}
\frac{1}{n}\sum_{t=1}^T\sum_{i=1}^n\E\Big[\big\|\gp(\x_t^i)\big\|^2
+ L^2\big\|\x_{t}^i - \ol{\x}_{t}\big\|^2\Big] 
\leq&~\frac{8\Delta}{\a}
+ \frac{440\lambda^2\zeta^2}{(1-\lambda^2)^3}
- \frac{1}{2}\left(1-\frac{21120\lambda^2\a^2L^2}{(1-\lambda^2)^4}\right)\sum_{t=1}^T\E\Big[\big\|\ol{\sgp}_{t}\big\|^2\Big]
.
\end{align*}
Hence, if $0<\a\leq\min\left\{\frac{1}{8},\sqrt{\frac{nb}{152q}},\frac{(1-\lambda^2)^2}{146\lambda^2}\right\}\frac{1}{L}$, then
\begin{align}\label{SRE_main}
\frac{1}{nT}\sum_{t=1}^T\sum_{i=1}^n\E\Big[\big\|\gp(\x_t^i)\big\|^2
+ L^2\big\|\x_{t}^i - \ol{\x}_{t}\big\|^2\Big] 
\leq&~\frac{8\Delta}{\a T}
+ \frac{440\lambda^2\zeta^2}{(1-\lambda^2)^3T}.
\end{align}
Let $\epsilon>0$ be given. Recall from~\eqref{lambda_K} that~$\lambda := \ltrue^K$ and we set 
\begin{align*}
K \asymp \frac{\log\zeta}{1-\ltrue},
\end{align*}
so that
$\frac{1}{1-\lambda} = \mc{O}(1),
\lambda\zeta = \mc{O}(1);$
moreover, we let
\begin{align}\label{SRE_qba}
q = \sqrt{nm},
\qquad
b = \max\left\{\sqrt{\frac{m}{n}},1\right\},
\qquad
\a\asymp\frac{1}{L}.
\end{align}
As a consequence, we have from~\eqref{SRE_main} that
\begin{align}\label{SRE_main_F}
\frac{1}{nT}\sum_{t=1}^T\sum_{i=1}^n\E\Big[\big\|\gp(\x_t^i)\big\|^2
+ L^2\big\|\x_{t}^i - \ol{\x}_{t}\big\|^2\Big] 
\lesssim \frac{L\Delta}{T}.
\end{align}
In view of~\eqref{SRE_main_F}, we further choose 
\begin{align}\label{SRE_T}
T \asymp \frac{L\Delta}{\epsilon^2} + q
\end{align}
which leads to $\frac{1}{nT}\sum_{t=1}^T\sum_{i=1}^n\E\big[\|\gp(\x_t^i)\|^2
+ L^2\|\x_{t}^i - \ol{\x}_{t}\|^2\big]
\lesssim \epsilon^2.$ The communication complexity is thus $TK$.
Since \texttt{ProxGT-SR-E} requires $m$ samples every $q$ iterations and $b$ samples at each iteration, its total sample complexity is bounded by 
\begin{align}\label{SRE_IFO}
\mc{O}\left(T\left(b + \frac{m}{q}\right)\right).
\end{align}
Plugging~\eqref{SRE_qba} and~\eqref{SRE_T} into~\eqref{SRE_IFO} concludes 
the ensuing sample complexity.

\section{Proof of Lemma~\ref{lem_descent}}\label{app_proof_descent_lem}
\subsection{Step 1: Descent Inequality for the Convex Part}
First of all, we write the proximal descent step in Algorithm~\ref{PGT} in an equivalent form for analysis purposes. For all~$t\geq1$ and~$i\in\mc{V}$, we observe that
\begin{align}\label{PG}
\z_{t+1}^i
=\p_{\alpha h}\big(\x_{t}^i - \alpha \y_{t+1}^i\big)  
=&~\argmin_{\u\in\R^p}\left\{\frac{1}{2}\big\|\u - (\x_{t}^i - \alpha \y_{t+1}^i)\big\|^2 + \alpha h(\u)\right\}       \nonumber\\
=&~\argmin_{\u\in\R^p}\left\{\frac{1}{2}\big\|\u - \x_{t}^i\big\|^2 
+ \bl \a\y_{t+1}^i,\u-\x_t^i\br
+ \frac{1}{2}\big\|\alpha \y_{t+1}^i\big\|^2
+ \alpha h(\u)\right\}       \nonumber\\
=&~\argmin_{\u\in\R^p}\left\{\bl \y_{t+1}^i,\u\br
+ \frac{1}{2\alpha}\big\|\u - \x_{t}^i\big\|^2 + h(\u)\right\}.
\end{align}
In light of the optimality condition of the strongly convex optimization problem~\eqref{PG} and the sum rule of subdifferential calculus~\cite[Theorem 3.40]{book_Beck2}, for all~$t\geq1$ and~$i\in\mc{V}$, there exists~$h'(\z_{t+1}^i)\in\partial h(\z_{t+1}^i)$ such that
\begin{align}\label{opt_PG}
h'(\z_{t+1}^i) = - \y_{t+1}^i - \frac{1}{\alpha}\big(\z_{t+1}^i - \x_{t}^i\big).
\end{align}
By the subgradient inequality, we have: $\forall t\geq1$, $\forall i\in\mc{V}$, and $\forall \u\in\R^p$,
\begin{align*}
h(\u) \geq h(\z_{t+1}^i) + \bl h'(\z_{t+1}^i), \u - \z_{t+1}^i\br,
\end{align*}
which is the same as
\begin{align}\label{subgrad_h}
h(\z_{t+1}^i) \leq h(\u) + \bl h'(\z_{t+1}^i), \z_{t+1}^i - \u\br.  
\end{align}
Applying~\eqref{opt_PG} to~\eqref{subgrad_h}, we obtain: $\forall t\geq1$, $\forall i\in\mc{V}$, and $\forall \u\in\R^p$,
\begin{align}\label{descent0_h}
h(\z_{t+1}^i) 
\leq
h(\u) 
- \frac{1}{\alpha}\left\langle \x_{t}^i - \z_{t+1}^i ,  \u - \z_{t+1}^i\right\rangle
- \left\langle \y_{t+1}^i, \z_{t+1}^i - \u\right\rangle.
\end{align}
We have the following algebraic identity: $\forall t\geq1$, $\forall i\in\mc{V}$, and $\forall \u\in\R^p$,
\begin{align}\label{three_points}
\left\langle \x_{t}^i - \z_{t+1}^i ,  \u - \z_{t+1}^i\right\rangle
= \frac{1}{2}\left\|\u - \z_{t+1}^i\right\|^2
+ \frac{1}{2}\left\|\x_{t}^i - \z_{t+1}^i\right\|^2
- \frac{1}{2}\left\|\x_{t}^i - \u\right\|^2.
\end{align}
Applying~\eqref{three_points} to~\eqref{descent0_h}, we obtain: $\forall t\geq1$, $\forall i\in\mc{V}$, and $\forall \u\in\R^p$,
\begin{align}\label{descent1_h}
h(\z_{t+1}^i) 
\leq
h(\u) 
-\frac{1}{2\alpha}\left\|\u - \z_{t+1}^i\right\|^2
- \frac{1}{2\alpha}\left\|\x_{t}^i - \z_{t+1}^i\right\|^2
+ \frac{1}{2\alpha}\left\|\x_{t}^i - \u\right\|^2
- \left\langle \y_{t+1}^i, \z_{t+1}^i - \u\right\rangle.
\end{align}
Setting~$\u:= \ol{\x}_t$, we have: $\forall t\geq1$ and $\forall i\in\mc{V}$, 
\begin{align}\label{descent2_h}
h(\z_{t+1}^i) 
\leq&~h(\ol{\x}_t) 
-\frac{1}{2\alpha}\left\|\ol{\x}_t - \z_{t+1}^i\right\|^2
- \frac{1}{2\alpha}\left\|\x_{t}^i - \z_{t+1}^i\right\|^2
+ \frac{1}{2\alpha}\left\|\x_{t}^i - \ol{\x}_t\right\|^2
- \left\langle \y_{t+1}^i, \z_{t+1}^i - \ol{\x}_t\right\rangle
\nonumber\\
=&~h(\ol{\x}_t) 
-\frac{1}{2\alpha}\left\|\ol{\x}_t - \z_{t+1}^i\right\|^2
- \frac{\alpha}{2}\left\|\sgp_{t}^{i}\right\|^2
+ \frac{1}{2\alpha}\left\|\x_{t}^i - \ol{\x}_t\right\|^2
- \left\langle \y_{t+1}^i - \ol{\y}_{t+1}, \z_{t+1}^i - \ol{\x}_t\right\rangle \nonumber\\
&- \left\langle \ol{\y}_{t+1}, \z_{t+1}^i - \ol{\x}_t\right\rangle,
\end{align}
where the last line uses~\eqref{sgp}.
For the second last term in~\eqref{descent2_h}, we have:~$\forall t\geq1$ and $\forall i\in\mc{V}$, 
\begin{align}\label{ip1_descent_h}
- \left\langle \y_{t+1}^i - \ol{\y}_{t+1}, \z_{t+1}^i - \ol{\x}_t\right\rangle
\leq&~\|\y_{t+1}^i - \ol{\y}_{t+1}\| \|\z_{t+1}^i - \ol{\x}_t\|
\nonumber\\
\leq&~\frac{\alpha}{2}\|\y_{t+1}^i - \ol{\y}_{t+1}\|^2
+ \frac{1}{2\alpha} \|\z_{t+1}^i - \ol{\x}_t\|^2,
\end{align}
where the first and the second line use the Cauchy-Schwarz and Young's inequality respectively. 
Plugging~\eqref{ip1_descent_h} into~\eqref{descent2_h} gives: $\forall t\geq1$ and $\forall i\in\mc{V}$, 
\begin{align}\label{descent3_h}
h(\z_{t+1}^i) 
\leq h(\ol{\x}_t) 
- \frac{\alpha}{2}\left\|\sgp_{t}^i\right\|^2
+ \frac{1}{2\alpha}\left\|\x_{t}^i - \ol{\x}_t\right\|^2
+ \frac{\alpha}{2}\|\y_{t+1}^i - \ol{\y}_{t+1}\|^2
- \left\langle \ol{\y}_{t+1}, \z_{t+1}^i - \ol{\x}_t\right\rangle.
\end{align}
We now average~\eqref{descent3_h} over~$i$ from~$1$ to~$n$ to obtain:~$\forall t\geq1$,
\begin{align}\label{descent5_h}
\frac{1}{n}\sum_{i=1}^n h(\z_{t+1}^i) 
\leq&~h(\ol{\x}_t) 
- \frac{\alpha}{2n}\sum_{i=1}^n\left\|\sgp_{t}^i\right\|^2
+ \frac{1}{2\alpha n}\left\|\x_{t} - \J\x_{t}\right\|^2
+ \frac{\alpha}{2n}\|\y_{t+1} - \J\y_{t+1}\|^2
- \big\langle \ol{\y}_{t+1}, \ol{\z}_{t+1} - \ol{\x}_t\big\rangle
\nonumber\\
=&~h(\ol{\x}_t) 
- \frac{\alpha}{2n}\sum_{i=1}^n\left\|\sgp_{t}^i\right\|^2
+ \frac{1}{2\alpha n}\left\|\x_{t} - \J\x_{t}\right\|^2
+ \frac{\alpha}{2n}\|\y_{t+1} - \J\y_{t+1}\|^2
+ \a\left\langle \ol{\y}_{t+1}, \ol{\sgp}_t\right\rangle,
\end{align}
where the second line follows from~\eqref{sgp_ave}. 
In light of the convexity of~$h$ and Jensen's inequality, for all~$t\geq1$ we have that $h(\ol{\z}_{t+1})\leq\frac{1}{n}\sum_{i=1}^n h(\z_{t+1}^i)$ and hence~\eqref{descent5_h}
implies
\begin{align}\label{descent6_h}
h(\ol{\z}_{t+1})
\leq h(\ol{\x}_t) 
- \frac{\alpha}{2n}\sum_{i=1}^n\left\|\sgp_{t}^i\right\|^2
+ \frac{1}{2\alpha n}\left\|\x_{t} - \J\x_{t}\right\|^2
+ \frac{\alpha}{2n}\|\y_{t+1} - \J\y_{t+1}\|^2
+ \alpha \left\langle \ol{\y}_{t+1}, \ol{\sgp}_{t}\right\rangle,
\quad\forall t\geq1.
\end{align}
In view of~\eqref{z_track}, we observe that~\eqref{descent6_h} is the same as
\begin{align}\label{descent_h}
h(\ol{\x}_{t+1})
\leq h(\ol{\x}_t) 
- \frac{\alpha}{2n}\sum_{i=1}^n\left\|\sgp_{t}^i\right\|^2
+ \frac{1}{2\alpha n}\left\|\x_{t} - \J\x_{t}\right\|^2
+ \frac{\alpha}{2n}\|\y_{t+1} - \J\y_{t+1}\|^2
+ \alpha \left\langle \ol{\y}_{t+1}, \ol{\sgp}_{t}\right\rangle, 
\quad\forall t\geq1.
\end{align}

\subsection{Step 2: Descent Inequality for the Non-Convex Part} 
Since~$F$ is~$L$-smooth, we have the standard quadratic upper bound~\cite[Lemma 5.7]{book_Beck2}:
\begin{align}\label{Lsmooth_F}
F(\y) \leq F(\x) + \bl \n F(\x), \y - \x\br
+ \frac{L}{2}\|\y-\x\|^2, \qquad\forall\x,\y\in\R^{p}.
\end{align}
Setting~$\y = \ol{\x}_{t+1}$ and~$\x = \ol{\x}_t$ in~\eqref{Lsmooth_F}, we obtain:~$\forall t\geq1$,
\begin{align}\label{descent_f}
F(\ol{\x}_{t+1}) 
\leq&~F(\ol{\x}_t) + \bl \n F(\ol{\x}_t), \ol{\x}_{t+1} - \ol{\x}_t\br
+ \frac{L}{2}\|\ol{\x}_{t+1}-\ol{\x}_{t}\|^2 \nonumber\\    
=&~F(\ol{\x}_t) - \alpha \bl \n F(\ol{\x}_t), \ol{\sgp}_{t}\br
+ \frac{L\alpha^2}{2}\|\ol{\sgp}_{t}\|^2,
\end{align}
where the last line is due to~\eqref{sgp_ave_d}.

\subsection{Step 3: Combining Step 1 and Step 2} 
Recall that~$\Psi:= F + h$. Summing up~\eqref{descent_f} and~\eqref{descent_h}, we obtain:~$\forall t\geq1$,
\begin{align}\label{descent_Psi1}
\Psi(\ol{\x}_{t+1}) 
\leq&~\Psi(\ol{\x}_t) 
- \frac{\alpha}{2n}\sum_{i=1}^n\left\|\sgp_{t}^i\right\|^2
+ \frac{1}{2\alpha n}\left\|\x_{t} - \J\x_{t}\right\|^2
+ \frac{\alpha}{2n}\|\y_{t+1} - \J\y_{t+1}\|^2 \nonumber\\
&+ \alpha \left\langle \ol{\y}_{t+1} - \n F(\ol{\x}_t), \ol{\sgp}_{t}\right\rangle
+ \frac{L\alpha^2}{2}\|\ol{\sgp}_{t}\|^2.
\end{align}
By the Cauchy-Schwarz and Young's inequality, we have: $\forall\eta>0$ and~$\forall t\geq1$,
\begin{align}\label{ip_descent_Psi1}
\left\langle \ol{\y}_{t+1} - \n F(\ol{\x}_t), \ol{\sgp}_{t}\right\rangle
\leq&~\frac{1}{2\eta}\left\|\ol{\y}_{t+1} - \n F(\ol{\x}_t)\right\|^2 + \frac{\eta}{2}\left\|\ol{\sgp}_{t}\right\|^2
\end{align}
Applying~\eqref{ip_descent_Psi1} to~\eqref{descent_Psi1}, we obtain: $\forall\eta>0$ and~$\forall t\geq1$,
\begin{align}\label{descent_Psi2}
\Psi(\ol{\x}_{t+1}) 
\leq&~\Psi(\ol{\x}_t) 
- \frac{\alpha}{2n}\sum_{i=1}^n\left\|\sgp_{t}^i\right\|^2
+ \frac{1}{2\alpha n}\left\|\x_{t} - \J\x_{t}\right\|^2
+ \frac{\alpha}{2n}\|\y_{t+1} - \J\y_{t+1}\|^2 \nonumber\\
&+ \frac{\alpha}{2\eta}\left\|\ol{\y}_{t+1} - \n F(\ol{\x}_t)\right\|^2 
+ \frac{\eta\alpha + L\alpha^2}{2}\|\ol{\sgp}_{t}\|^2.
\end{align}

\noindent
\subsection{Step 4: Refining Error Terms and Telescoping Sum}
We first bound the difference between the local stochastic gradient mapping~$\sgp_{t}^i$ defined in~\eqref{sgp} and the exact gradient mapping~$\gp(\x_t^i)$ defined in~\eqref{GP}. Observe that~$\forall t\geq1$ and~$\forall i\in\mc{V}$,
\begin{align}\label{map_diff}
\left\| \sgp_{t}^i - \gp(\x_t^i)\right\|^2
=&~\left\|\frac{1}{\a}\Big(\x_t^i - \p_{\alpha h}\big(\x_{t}^i - \alpha \y_{t+1}^i\big)\Big) 
- \frac{1}{\a}\Big(\x_t^i - \p_{\alpha h}\big(\x_t^i - \alpha \n F(\x_t^i)\big)\Big)\right\|^2
\nonumber\\
=&~\frac{1}{\alpha^2}\Big\|\p_{\alpha h}\big(\x_{t}^i - \alpha \y_{t+1}^i\big) - \p_{\alpha h}\big(\x_t^i - \alpha \n F(\x_t^i)\big)\Big\|^2 
\nonumber\\
\leq&~\big\|\y_{t+1}^i - \n F(\x_t^i)\big\|^2  
\nonumber\\
=&~\big\|\y_{t+1}^i - \ol{\y}_{t+1} + \ol{\y}_{t+1} - \n F(\ol{\x}_t) + \n F(\ol{\x}_t) - \n F(\x_t^i)\big\|^2  
\nonumber\\
\leq&~3\big\|\y_{t+1}^i - \ol{\y}_{t+1}\big\|^2
+3\big\|\ol{\y}_{t+1} - \n F(\ol{\x}_t)\big\|^2
+3L^2\big\|\ol{\x}_t - \x_t^i\big\|^2,
\end{align}
where the third line is due to Lemma~\ref{lem_prox_nonexp} and the last line uses the~$L$-smoothness of~$F$. Observe that~$\forall t\geq1$,
\begin{align}\label{map_diff2}
- \big\|\sgp_{t}^i\big\|^2
\leq& -\frac{1}{2}\big\|\gp(\x_t^i)\big\|^2 + \big\| \sgp_{t}^i - \gp(\x_t^i) \big\|^2
\nonumber\\
\leq& -\frac{1}{2}\big\|\gp(\x_t^i)\big\|^2 
+ 3\big\|\y_{t+1}^i - \ol{\y}_{t+1}\big\|^2
+3\big\|\ol{\y}_{t+1} - \n F(\ol{\x}_t)\big\|^2
+3L^2\big\|\ol{\x}_t - \x_t^i\big\|^2,
\end{align} 
where the first line is due to the standard triangular inequality and the second line uses~\eqref{map_diff}. Averaging~\eqref{map_diff2} over~$i$ from~$1$ to~$n$ gives:~$\forall t\geq1$,
\begin{align}\label{map_diff3}
- \frac{1}{n}\sum_{i=1}^n\left\|\sgp_{t}^i\right\|^2
\leq - \frac{1}{2n}\sum_{i=1}^n\left\|\gp(\x_t^i)\right\|^2 
+\frac{3}{n}\left\|\y_{t+1} - \J\y_{t+1}\right\|^2
+3\left\|\ol{\y}_{t+1} - \n F(\ol{\x}_t)\right\|^2
+ \frac{3L^2}{n}\big\|\x_t - \J\x_t\big\|^2.
\end{align} 
We now plug~\eqref{map_diff3} into~\eqref{descent_Psi2} to obtain:~$\forall\eta>0$ and~$\forall t\geq1$,
\begin{align}\label{descent_Psi3}
\Psi(\ol{\x}_{t+1}) 
\leq&~\Psi(\ol{\x}_t) 
- \frac{\alpha}{4n}\sum_{i=1}^n\left\|\sgp_{t}^i\right\|^2
- \frac{\alpha}{8n}\sum_{i=1}^n\left\|\gp(\x_t^i)\right\|^2
+ \left(\frac{1}{2\alpha}+\frac{3\a L^2}{4}\right)\frac{1}{n}\left\|\x_{t} - \J\x_{t}\right\|^2
+ \frac{5\alpha}{4n}\|\y_{t+1} - \J\y_{t+1}\|^2 \nonumber\\
&+ \left(\frac{3}{2} + \frac{1}{\eta}\right)\frac{\a}{2}\left\|\ol{\y}_{t+1} - \n F(\ol{\x}_t)\right\|^2 
+ \frac{\eta\alpha + L\alpha^2}{2}\|\ol{\sgp}_{t}\|^2 \nonumber\\
\leq&~\Psi(\ol{\x}_t) 
- \frac{\alpha - 2\eta\alpha - 2L\a^2}{4n}\sum_{i=1}^n\left\|\sgp_{t}^i\right\|^2
- \frac{\alpha}{8n}\sum_{i=1}^n\left\|\gp(\x_t^i)\right\|^2
+ \left(\frac{1}{2\alpha}+\frac{3\a L^2}{4}\right)\frac{1}{n}\left\|\x_{t} - \J\x_{t}\right\|^2
\nonumber\\
&+ \frac{5\alpha}{4n}\|\y_{t+1} - \J\y_{t+1}\|^2
+ \left(\frac{3}{2} + \frac{1}{\eta}\right)\frac{\a}{2}\left\|\ol{\y}_{t+1} - \n F(\ol{\x}_t)\right\|^2, 
\end{align}
where the last line is due to~$\|\ol{\sgp}_{t}\|^2 \leq \frac{1}{n}\sum_{i=1}^n\|\sgp_{t}^i\|^2.$
Setting~$\eta = \frac{1}{8}$ and~$0<\alpha\leq \frac{1}{8L}$ in~\eqref{descent_Psi3},
we have:~$\forall t\geq1$,
\begin{align}\label{descent_Psi3'}
\Psi(\ol{\x}_{t+1}) 
\leq&~\Psi(\ol{\x}_t) 
- \frac{\alpha}{8n}\sum_{i=1}^n\left\|\sgp_{t}^i\right\|^2
- \frac{\alpha}{8n}\sum_{i=1}^n\left\|\gp(\x_t^i)\right\|^2
+ \left(\frac{1}{2\alpha}+\frac{3\a L^2}{4}\right)\frac{1}{n}\left\|\x_{t} - \J\x_{t}\right\|^2
\nonumber\\
&+ \frac{5\alpha}{4n}\|\y_{t+1} - \J\y_{t+1}\|^2
+ \frac{19\a}{4}\left\|\ol{\y}_{t+1} - \n F(\ol{\x}_t)\right\|^2.
\end{align}
Towards the last term in~\eqref{descent_Psi3'}, observe that,~$\forall t\geq1$,
\begin{align}\label{descent_Psi3_y_track}
\big\|\ol{\y}_{t+1} - \n F(\ol{\x}_t)\big\|^2
=&~\|\ol{\mb{v}}_{t} - \n F(\ol{\x}_t)\|^2 \nonumber\\
\leq&~2\big\|\ol{\mb{v}}_{t} - \ol{\nf}(\x_t)\big\|^2
+ 2\big\|\ol{\nf}(\x_t) - \n F(\ol{\x}_t)\big\|^2 \nonumber\\
\leq&~2\big\|\ol{\mb{v}}_{t} - \ol{\nf}(\x_t)\big\|^2
+ 2L^2n^{-1}\big\|\x_t - \J\x_t\big\|^2,
\end{align}
where the first line is due to~\eqref{y_track} while the last line uses the~$L$-smoothness of each~$f_i$, i.e., 
\begin{align*}
\big\|\ol{\nf}(\x_t) - \n F(\ol{\x}_t)\big\|^2
= \left\|\frac{1}{n}\sum_{i=1}^n\Big(\n f_i(\x_t^i) - \n f_i(\ol{\x}_t)\Big)\right\|^2
\leq \frac{1}{n}\sum_{i=1}^n\big\|\n f_i(\x_t^i) - \n f_i(\ol{\x}_t)\big\|^2
\leq \frac{L^2}{n}\big\|\x_t - \J\x_t\big\|^2.
\end{align*}
Plugging~\eqref{descent_Psi3_y_track} into~\eqref{descent_Psi3'}, we have: if~$0<\a\leq\frac{1}{8L}$, then~$\forall t\geq1$,
\begin{align}\label{descent_Psi4}
\Psi(\ol{\x}_{t+1}) 
\leq&~\Psi(\ol{\x}_t) 
- \frac{\alpha}{8n}\sum_{i=1}^n\left\|\sgp_{t}^i\right\|^2
- \frac{\alpha}{8n}\sum_{i=1}^n\left\|\gp(\x_t^i)\right\|^2
+ \left(\frac{1}{2\alpha}+\frac{41\a L^2}{4}\right)\frac{1}{n}\left\|\x_{t} - \J\x_{t}\right\|^2
\nonumber\\
&+ \frac{5\alpha}{4n}\|\y_{t+1} - \J\y_{t+1}\|^2
+ \frac{19\a}{2}\left\|\ol{\mb{v}}_{t} - \ol{\nf}(\x_t)\right\|^2.  
\end{align}
Telescoping sum~\eqref{descent_Psi4} over~$t$ from~$1$ to~$T$, we have:
if~$0<\a\leq\frac{1}{8L}$, then
\begin{align}\label{descent_Psi5}
\Psi(\ol{\x}_{T+1}) 
\leq&~\Psi(\ol{\x}_1) 
- \frac{\alpha}{8n}\sum_{t=1}^T\sum_{i=1}^n\left\|\sgp_{t}^i\right\|^2
- \frac{\alpha}{8n}\sum_{t=1}^T\sum_{i=1}^n\left\|\gp(\x_t^i)\right\|^2
+ \left(\frac{1}{2\alpha}+\frac{41\a L^2}{4}\right)\frac{1}{n}\sum_{t=1}^T\left\|\x_{t} - \J\x_{t}\right\|^2
\nonumber\\
&+ \frac{5\alpha}{4n}\sum_{t=1}^T\|\y_{t+1} - \J\y_{t+1}\|^2
+ \frac{19\a}{2}\sum_{t=1}^T\left\|\ol{\mb{v}}_{t} - \ol{\nf}(\x_t)\right\|^2.  
\end{align}
With~$\inf_{\x\in\R^p}\Psi(\x)\geq\ul{\Psi}>-\infty$ and minor rearrangement,~\eqref{descent_Psi5} implies the following: if~$0<\a\leq\frac{1}{8L}$, then
\begin{align*}
\frac{1}{n}\sum_{t=1}^T\left(\sum_{i=1}^n\left\|\gp(\x_t^i)\right\|^2
+ L^2\left\|\x_{t} - \J\x_{t}\right\|^2\right) 
\leq&~\frac{8(\Psi(\ol{\x}_1) - \ul{\Psi})}{\a}
- \frac{1}{n}\sum_{t=1}^T\sum_{i=1}^n\left\|\sgp_{t}^i\right\|^2
+ 76\sum_{t=1}^T\left\|\ol{\mb{v}}_{t} - \ol{\nf}(\x_t)\right\|^2
\nonumber\\
&+ \left(\frac{4}{\alpha^2}+83L^2\right)\frac{1}{n}\sum_{t=1}^T\left\|\x_{t} - \J\x_{t}\right\|^2
+ \frac{10}{n}\sum_{t=2}^{T+1}\|\y_{t} - \J\y_{t}\|^2,  
\end{align*}
which finishes the proof of Lemma~\ref{lem_descent} by~$83L^2\leq\frac{2}{\a^2}$.

\section{Proof of Lemma~\ref{lem_consensus}}\label{app_proof_consensus_lem}
For ease of exposition, we define a block-wise proximal mapping for~$h$:
\begin{align}\label{prox_vec}
\bp(\mb{c})
:= 
\begin{bmatrix}
\p_{\alpha h}(\mb{c}_1) \\[0.2em]
\cdots \\
\p_{\alpha h}(\mb{c}_n) 
\end{bmatrix}
\in\R^{np},
\quad
\text{where}
\quad
\mb{c}:= 
\begin{bmatrix}
\mb{c}_1 \\[0.2em]
\cdots \\
\mb{c}_n
\end{bmatrix}
\end{align}
such that~$\mb{c}_i\in\R^p,\forall i\in[n]$. In view of~\eqref{prox_vec}, the~$\x$-update in Algorithm~\ref{PGT} can compactly be written as
\begin{align}\label{prox_x_vec}
\x_{t+1} = \W^K\bp(\x_t - \alpha \y_{t+1}),
\qquad
\forall t\geq1.
\end{align}
We find the following quantity helpful:~$\forall t\geq1$,
\begin{align}\label{consensus1}
\big(\W^{K}-\J\big)\bp(\J\x_t - \alpha\J\y_{t+1}) 
=&~\Big(\big(\Wtrue^K-\tfrac{1}{n}\mb{1}_n\mb{1}_n^\top\big)\otimes\I_p\Big)\Big(\mb{1}_n \otimes\p_{\alpha h}\big(\ol{\x}_t-\alpha\ol{\y}_{t+1}\big)\Big) \nonumber\\
=&~\Big(\big(\Wtrue^K-\tfrac{1}{n}\mb{1}_n\mb{1}_n^\top\big)\mb{1}_n\Big)\otimes\p_{\alpha h}\big(\ol{\x}_t-\alpha\ol{\y}_{t+1}\big) \nonumber\\
=&~\mb{0}_{np},    
\end{align}
where the first line uses the definition of~$\W$,~$\J$, and~$\bp$, and the last line is due to the doubly stochasticity of~$\Wtrue$. We are now prepared to analyze the consensus error recursion in the following. For all~$t\geq1$, we have
\begin{align}\label{consensus2}
\big\|\x_{t+1}-\J\x_{t+1}\big\|^2
=&~\Big\|\W^K\bp(\x_t - \alpha \y_{t+1})-\J\W^K\bp(\x_t - \alpha \y_{t+1})\Big\|^2 \nonumber\\
=&~\Big\|\big(\W^K-\J\big)\bp(\x_t - \alpha \y_{t+1})\Big\|^2      \nonumber\\
=&~\Big\|\big(\W^K-\J\big)\Big(\bp(\x_t - \alpha \y_{t+1}) - \bp(\J\x_t - \alpha \J\y_{t+1})\Big)\Big\|^2
\nonumber\\
\leq&~\lambda^2\Big\|\bp(\x_t - \alpha \y_{t+1}) - \bp(\J\x_t - \alpha \J\y_{t+1})\Big\|^2,
\end{align}
where the first line uses~\eqref{prox_x_vec}, the second line follows from Lemma~\ref{lem_weight}\eqref{W_equal}, the third line is due to~\eqref{consensus1}, and the last line uses Lemma~\ref{lem_weight}\eqref{W_spectral}. 
To proceed from~\eqref{consensus2}, we observe that~$\forall t\geq1$,
\begin{align}\label{block_nonexp}
\Big\|\bp(\x_t - \alpha \y_{t+1}) - \bp(\J\x_t - \alpha \J\y_{t+1})\Big\|^2 
=&~\sum_{i=1}^n\Big\|\p_{\alpha h}(\x_t^i - \alpha\y_{t+1}^i) - \p_{\alpha h}(\ol{\x}_t - \alpha\ol{\y}_{t+1})\Big\|^2     \nonumber\\
\leq&~\sum_{i=1}^n\big\|\x_t^i - \ol{\x}_t - \alpha(\y_{t+1}^i - \ol{\y}_{t+1})\big\|^2
\nonumber\\
=&~\big\|\x_t - \J\x_t - \alpha(\y_{t+1} - \J\y_{t+1})\big\|^2,
\end{align}
where the first and the second line uses~\eqref{prox_vec} and Lemma~\ref{lem_prox_nonexp} respectively. We then plug~\eqref{block_nonexp} into~\eqref{consensus2} to obtain:~$\forall t\geq1$ and~$\forall \eta>0$,
\begin{align}\label{consensus3}
\big\|\x_{t+1}-\J\x_{t+1}\big\|^2
\leq&~\lambda^2\big\|\x_t - \J\x_t - \alpha(\y_{t+1} - \J\y_{t+1})\big\|^2
\nonumber\\
=&~\lambda^2\big\|\x_t - \J\x_t\big\|^2 + \lambda^2\alpha^2\big\|\y_{t+1} - \J\y_{t+1}\big\|^2
- 2\lambda^2\big\langle \x_t - \J\x_t, \alpha(\y_{t+1} - \J\y_{t+1})\big\rangle
\nonumber\\
\leq&~\lambda^2\big\|\x_t - \J\x_t\big\|^2 + \lambda^2\alpha^2\big\|\y_{t+1} - \J\y_{t+1}\big\|^2
+ 2\lambda^2\big\|\x_t - \J\x_t\big\| \big\|\alpha(\y_{t+1} - \J\y_{t+1})\big\|
\nonumber\\
\leq&~\lambda^2\big(1+\eta\big)\big\|\x_t - \J\x_t\big\|^2 + \lambda^2\alpha^2\big(1+\eta^{-1}\big)\big\|\y_{t+1} - \J\y_{t+1}\big\|^2,
\end{align}
where the third and the last line use the Cauchy-Schwarz and Young's inequality with parameter~$\eta$ respectively. Finally, setting~$\eta = \frac{1-\lambda^2}{2\lambda^2}$ in~\eqref{consensus3} yields:~$\forall t\geq1$,
\begin{align}\label{consensus4}
\big\|\x_{t+1}-\J\x_{t+1}\big\|^2
\leq \frac{1+\lambda^2}{2}\big\|\x_t - \J\x_t\big\|^2 + \frac{\lambda^2\alpha^2(1+\lambda^2)}{1-\lambda^2}\big\|\y_{t+1} - \J\y_{t+1}\big\|^2.
\end{align}
Applying Lemma~\ref{lem_acc} to~\eqref{consensus4}, we have:~$\forall T\geq2$,
\begin{align*}
\sum_{t=1}^T\big\|\x_{t}-\J\x_{t}\big\|^2
\leq  
\frac{2\lambda^2\alpha^2(1+\lambda^2)}{(1-\lambda^2)^2}\sum_{t=1}^{T-1}\big\|\y_{t+1} - \J\y_{t+1}\big\|^2,
\end{align*}
which finishes the proof of Lemma~\ref{lem_consensus}.

\section{Proof of Lemma~\ref{lem_VR}}\label{app_proof_VR_lem}

\subsection{Proof of Lemma~\ref{lem_VR}(\ref{lem_VR_M})}
We first recall that the gradient estimator~$\mb{v}_t^i$ in Algorithm~\ref{PGT_SA} takes the following form:~$\forall t\geq1$ and~$i\in\mc{V}$,
\begin{align*}
\mb{v}_t^i:= 
\frac{1}{b}\sum_{s=1}^b\n G_i\big(\x_t^i,\X_{i,s}^t\big).    
\end{align*}
Observe that~$\forall t\geq1$,
\begin{align*}
\E\Big[\big\|\ol{\mb{v}}_t - \ol{\nf}(\x_t)\big\|^2\big|\F_t\Big]
=&~\E\left[\left\|\frac{1}{nb}\sum_{i=1}^n\sum_{s=1}^b\Big(\n G_i\big(\x_t^i,\X_{i,s}^t\big) - \n f_i(\x_t^i)\Big)\right\|^2\Bigg|\F_t\right]
\nonumber\\
=&~\frac{1}{(nb)^2}\sum_{i=1}^n\sum_{s=1}^b\E\left[\left\|
\n G_i\big(\x_t^i,\X_{i,s}^t\big) - \n f_i(\x_t^i)\right\|^2\big|\F_t\right]
\nonumber\\
\leq&~\frac{1}{(nb)^2}\sum_{i=1}^n\sum_{s=1}^b\nu_i^2
\nonumber\\
=&~\frac{\nu^2}{nb},
\end{align*}
where the second line uses Assumption~\ref{asp_unbias} and the fact that~$\x_t$ is~$\F_t$-measurable and~$\{\X_{i,s}^t:i\in\mc{V},s\in[b]\}$ is independent of~$\F_t$,
while the third line is due to Assumption~\ref{asp_bvr}.

\subsection{Proof of Lemma~\ref{lem_VR}(\ref{lem_VR_O}) and Lemma~\ref{lem_VR}(\ref{lem_VR_E})}
To facilitate the analysis, we first note that the gradient estimator~$\mb{v}_t^i$ in both Algorithm~\ref{PGT_SR_O} and~\ref{PGT_SR_E} take the following form: $\forall i\in\mc{V}$ and $\forall t\geq1$ such that $\bmod(t,q) \neq 1$,
\begin{align*}
\mb{v}_t^i:= 
\frac{1}{b}\sum_{s=1}^{b}\Big(\n G_i(\x_t^i,\X_{i,s}^t) - \n G_i(\x_{t-1}^i,\X_{i,s}^t)\Big) + \mb{v}_{t-1}^i.
\end{align*}
To simplify notation, we denote in this section that $$\delta_t := \big\|\ol{\mb{v}}_t - \ol{\nf}(\x_t)\big\|^2, 
\qquad\forall t\geq1.$$
We establish an upper bound on $\delta_t$ that is applicable to both
Algorithm~\ref{PGT_SR_O} and~\ref{PGT_SR_E}.
\begin{lem}\label{lem_VR_SARAH}
Let Assumption~\ref{asp_mss} hold. Suppose that $T = Rq$ for some $R\in\mathbb{Z}^{+}$.
Consider the iterates generated by Algorithm~\ref{PGT_SR_O} or~\ref{PGT_SR_E}.  
Then we have: $\forall T\geq q$,
\begin{align*}
\sum_{t=1}^T\E\big[\delta_t\big]
\leq
\frac{6L^2q}{n^2b}\sum_{t=1}^T\E\Big[\big\|\mb{x}_t - \J\mb{x}_{t}\big\|^2\Big]
+ \frac{3qL^2\a^2}{nb}\sum_{t=1}^{T-1}\E\Big[\big\|\ol{\sgp}_{t}\big\|^2\Big]
+ q\sum_{z=1}^{R}\E\big[\delta_{(z-1)q+1}\big].
\end{align*}
\end{lem}
\begin{proof}
Consider any $t\geq1$ such that $\bmod(t,q) \neq 1$. For convenience, we define:~$\forall i\in\mc{V}$,
\begin{align*}
\mb{d}_t^{i,s} :
= \n G_i(\x_t^i,\X_{i,s}^t) - \n G_i(\x_{t-1}^i,\X_{i,s}^t),
\qquad
\mb{d}_t^{i} := \frac{1}{b}\sum_{s=1}^{b}\mb{d}_t^{i,s},
\end{align*}
and we clearly have
\begin{align}\label{diff_exp}
\E\big[\mb{d}_t^{i,s}|\F_t\big]    
= \E\big[\mb{d}_t^{i}|\F_t\big]    
= \n f_i(\x_t^i) - \n f_i(\x_{t-1}^i).    
\end{align}
As a consequence of~\eqref{diff_exp} and of the independence between $\bxi^{t}_{i,s}$ and $\F_t$ for all $i\in\mc{V}$ and $s\in[b]$, we have
\begin{align}\label{ip_indep_i}
\E\Big[\Big\langle \mb{d}_t^{i} - \n f_i(\x_t^i) + \n f_i(\x_{t-1}^i), \mb{d}_t^{r} - \n f_r(\x_t^r) + \n f_r(\x_{t-1}^r) \Big\rangle\big|\F_t\Big] = 0,    
\end{align}
whenever~$i \neq r$, and
\begin{align}\label{ip_indep_s}
\E\Big[\Big\langle \mb{d}_t^{i,s} - \n f_i(\x_t^i) + \n f_i(\x_{t-1}^i), \mb{d}_t^{i,a} - \n f_i(\x_t^i) + \n f_i(\x_{t-1}^r) \Big\rangle\big|\F_t\Big] = 0,    
\end{align}
whenever~$s\neq a$. Moreover, using the conditional variance decomposition with~\eqref{diff_exp} gives: $\forall i\in\mc{V}$ and $s\in[b]$,
\begin{align}\label{var_decom}
\E\Big[\big\|\mb{d}_t^{i,s} - \n f_i(\x_t^i) + \n f_i(\x_{t-1}^i)\big\|^2|\F_t\Big] 
\leq \E\Big[\big\|\mb{d}_t^{i,s}\big\|^2|\F_t\Big]. 
\end{align}
By the update of~$\mb{v}_t^i$, we observe that
\begin{align}\label{v1}
\E\big[\delta_t|\F_t\big]  
=&~\E\!\left[\Bigg\|\frac{1}{n}\sum_{i=1}^n\Big(\mb{d}_t^{i}+\mb{v}_{t-1}^i - \n f_i(\x_t^i)\Big)\Bigg\|^2\bigg|\F_t\right]   \nonumber\\
=&~\E\!\left[\Bigg\|\frac{1}{n}\sum_{i=1}^n\Big(\mb{d}_t^{i} - \n f_i(\x_t^i) + \n f_i(\x_{t-1}^i)  +\mb{v}_{t-1}^i - \n f_i(\x_{t-1}^i) \Big)\Bigg\|^2 \bigg|\F_t\right] \nonumber\\
=&~\E\!\left[\Bigg\|\frac{1}{n}\sum_{i=1}^n\Big(\mb{d}_t^{i} - \n f_i(\x_t^i) + \n f_i(\x_{t-1}^i)\Big)\Bigg\|^2 \Big|\F_t\right] 
+ \delta_{t-1}
\nonumber\\
=&~\frac{1}{n^2}\sum_{i=1}^n\E\!\left[\Big\|\mb{d}_t^{i} - \n f_i(\x_t^i) + \n f_i(\x_{t-1}^i)\Big\|^2 \Big|\F_t\right] 
+ \delta_{t-1}
\nonumber\\
=&~\frac{1}{n^2b^2}\sum_{i=1}^n\sum_{s=1}^b\E\!\left[\Big\|\mb{d}_t^{i,s} - \n f_i(\x_t^i) + \n f_i(\x_{t-1}^i)\Big\|^2 \Big|\F_t\right] 
+ \delta_{t-1}
\nonumber\\
\leq&~\frac{1}{n^2b^2}\sum_{i=1}^n\sum_{s=1}^b\E\!\left[\big\|\mb{d}_t^{i,s}\big\|^2 \big|\F_t\right] 
+ \delta_{t-1},
\end{align}
where the third line uses~\eqref{diff_exp}, the fourth line uses~\eqref{ip_indep_i}, the fifth line uses \eqref{ip_indep_s}, and the last line uses~\eqref{var_decom}. We note that the mean-squared smoothness of $\n G(\cdot)$ implies that for all~$i\in\mc{V}$ and $s\in[b]$,
\begin{align}\label{mss}
\E\Big[\big\|\mb{d}_t^{i,s}\big\|^2\Big]     
\leq L^2\E\Big[\big\|\mb{x}_t^i - \mb{x}_{t-1}^i\big\|^2\Big]. 
\end{align}
Applying~\eqref{mss} to~\eqref{v1} gives: for all $t\geq1$ such that $\bmod(t,q) \neq 1$, 
\begin{align}\label{v2}
\E\big[\delta_t\big] 
\leq \frac{L^2}{n^2b}\E\Big[\big\|\mb{x}_t - \mb{x}_{t-1}\big\|^2\Big] 
+ \E\big[\delta_{t-1}\big].
\end{align}
For convenience, we define
\begin{align*}
\vp_t 
:= \left\lfloor \frac{t-1}{q} \right\rfloor,
\qquad \forall t\geq1. 
\end{align*}
It can be verified that
\begin{align*}
\vp_t q + 1 \leq t \leq (\vp_t + 1)q,
\qquad \forall t\geq1. 
\end{align*}
With the help of the above notations, we recursively apply~\eqref{v2} from~$t$ to $(\vp_t q + 2)$ to obtain: for all $t\geq1$ such that $\bmod(t,q) \neq 1$, 
\begin{align}\label{v3}
\E\big[\delta_t\big]
\leq \frac{L^2}{n^2b}\sum_{j=\vp_tq+2}^{t}\E\Big[\big\|\mb{x}_j - \mb{x}_{j-1}\big\|^2\Big] 
+ \E\big[\delta_{\vp_t q+1}\big].
\end{align}
Summing up~\eqref{v3}, we observe that $\forall z\geq1$,
\begin{align}\label{v4}
\sum_{t=(z-1)q+1}^{zq}\E\big[\delta_t\big]
\leq&~\sum_{t=(z-1)q+2}^{zq}\left(\frac{L^2}{n^2b}\sum_{j=\vp_tq+2}^{t}\E\Big[\big\|\mb{x}_j - \mb{x}_{j-1}\big\|^2\Big]
+ \E\big[\delta_{\vp_t q+1}\big]\right) 
+ \E\big[\delta_{(z-1)q+1}\big]       \nonumber\\
=&~\frac{L^2}{n^2b}\sum_{t=(z-1)q+2}^{zq}\sum_{j=(z-1)q+2}^{t}\E\Big[\big\|\mb{x}_j - \mb{x}_{j-1}\big\|^2\Big]
+ q\E\big[\delta_{(z-1)q+1}\big] \nonumber\\
\leq&~\frac{L^2}{n^2b}\sum_{t=(z-1)q+2}^{zq}\sum_{j=(z-1)q+2}^{zq}\E\Big[\big\|\mb{x}_j - \mb{x}_{j-1}\big\|^2\Big]
+ q\E\big[\delta_{(z-1)q+1}\big] \nonumber\\
=&~\frac{L^2(q-1)}{n^2b}\sum_{j=(z-1)q+2}^{zq}\E\Big[\big\|\mb{x}_j - \mb{x}_{j-1}\big\|^2\Big]
+ q\E\big[\delta_{(z-1)q+1}\big],
\end{align}
where the second line uses the fact that~$\vp_t = z-1$ when $(z-1)q+1\leq t\leq zq$ for all $z\geq1$. Finally, we sum up~\eqref{v4} over $z$ from~$1$ to $R$, we obtain: $\forall R\geq1$,
\begin{align}\label{v5}
\sum_{z=1}^{R}\sum_{t=(z-1)q+1}^{zq}\E\big[\delta_t\big]
\leq
\frac{L^2(q-1)}{n^2b}\sum_{z=1}^{R}\sum_{j=(z-1)q+2}^{zq}\E\Big[\big\|\mb{x}_j - \mb{x}_{j-1}\big\|^2\Big]
+ q\sum_{z=1}^{R}\E\big[\delta_{(z-1)q+1}\big].
\end{align}
Recall that $T = Eq$ and from~\eqref{v5} we obtain that $\forall T\geq q$,
\begin{align}\label{v6}
\sum_{t=1}^T\E\big[\delta_t\big]
\leq
\frac{L^2(q-1)}{n^2b}\sum_{t=2}^T\E\Big[\big\|\mb{x}_t - \mb{x}_{t-1}\big\|^2\Big]
+ q\sum_{z=1}^{R}\E\big[\delta_{(z-1)q+1}\big].
\end{align}
Finally, we apply Lemma~\ref{lem_x_diff} to~\eqref{v6} to obtain:~$\forall T\geq q$,
\begin{align*}
\sum_{t=1}^T\E\big[\delta_t\big]
\leq&~\frac{6L^2(q-1)}{n^2b}\sum_{t=1}^T\E\big[\|\mb{x}_t - \J\mb{x}_{t}\|^2\big]
+ \frac{3L^2(q-1)\a^2}{nb}\sum_{t=1}^{T-1}\E\Big[\big\|\ol{\sgp}_{t}\big\|^2\Big] 
+ q\sum_{z=1}^{R}\E\big[\delta_{(z-1)q+1}\big],
\end{align*}
which finishes the proof.
\end{proof}

\subsubsection{Proof of Lemma~\ref{lem_VR}(\ref{lem_VR_O})}
Lemma~\ref{lem_VR}(\ref{lem_VR_E}) follows by applying Lemma~\ref{lem_VR}\eqref{lem_VR_M} to Lemma~\ref{lem_VR_SARAH}, i.e.,
$\E[\delta_{(z-1)q+1}] \leq \frac{\nu^2}{nB}$ for all $z\geq1$.

\subsubsection{Proof of Lemma~\ref{lem_VR}(\ref{lem_VR_E})}
Lemma~\ref{lem_VR}(\ref{lem_VR_O}) follows from Lemma~\ref{lem_VR_SARAH} by $\delta_{(z-1)q+1} = 0$ for all $z\geq1$.

\section{Proof of Lemma~\ref{lem_GT}}\label{app_proof_GT_lem}
We first present a simple result that is useful for our later development.
\begin{prop}
Consider the iterates generated by Algorithm~\ref{PGT}.
The following inequality holds:~$\forall t\geq1$,
\begin{align}\label{GT_0}
\|\y_{t+1} -\mb{J}\y_{t+1}\|^2 
\leq \lambda^2\| \mb{y}_{t}-\mb{J}\mb{y}_{t}\|^2
+ \lambda^2\|\mb{v}_{t} - \mb{v}_{t-1}\|^2 
+ 2\big\langle\mb{W}^K\mb{y}_{t}-\mb{J}\mb{y}_{t},\big(\mb{W}^K -\mb{J}\big)\big(\mb{v}_{t} - \mb{v}_{t-1}\big)\big\rangle.
\end{align}
\end{prop}
\begin{proof}
Using the $\mb{y}$-update in Algorithm~\ref{PGT} and Lemma~\ref{lem_weight}(\ref{W_equal}), we have:~$\forall t\geq1$,
\begin{align}
&\|\y_{t+1} -\mb{J}\y_{t+1}\|^2  \nonumber\\
=&~\big\| \mb{W}^K\big(\mb{y}_{t} + \mb{v}_{t} - \mb{v}_{t-1}\big) -\mb{J}\mb{W}^K\big(\mb{y}_{t} + \mb{v}_{t} - \mb{v}_{t-1}\big)\big\|^2 \nonumber\\
=&~\big\| \mb{W}^K\mb{y}_{t}-\mb{J}\mb{y}_{t} + \big(\mb{W}^K -\mb{J}\big)\big(\mb{v}_{t} - \mb{v}_{t-1}\big)\big\|^2 \nonumber\\
=&~\big\|\mb{W}^K\mb{y}_{t}-\mb{J}\mb{y}_{t}\big\|^2 +
\big\|\big(\mb{W}^K -\mb{J}\big)\big(\mb{v}_{t} - \mb{v}_{t-1}\big)\big\|^2 
+ 2\big\langle \mb{W}^K\mb{y}_{t} - \mb{J}\mb{y}_{t}, \big(\mb{W}^K -\mb{J}\big)\big(\mb{v}_{t} - \mb{v}_{t-1}\big) \big\rangle,\nonumber
\end{align}
and the proof follows by using Lemma~\ref{lem_weight}(\ref{W_contract}).
\end{proof}

\subsection{Proof of Lemma~\ref{lem_GT}(\ref{lem_GT_M})}
\subsubsection{Step 1: Decomposition}
Recall that we are concerned with Algorithm~\ref{PGT_SA} in this section. Conditioning~\eqref{GT_0} on~$\F_t$, we have:~$\forall t\geq2$, 
\begin{align}\label{SA_GT_0}
&\E\big[\|\y_{t+1} -\mb{J}\y_{t+1}\|^2|\F_t\big] \nonumber\\
\leq&~\lambda^2\| \mb{y}_{t}-\mb{J}\mb{y}_{t}\|^2
+ \lambda^2\E\big[\|\mb{v}_{t} - \mb{v}_{t-1}\|^2|\F_t\big]  
+ 2\big\langle\mb{W}^K\mb{y}_{t}-\mb{J}\mb{y}_{t},\big(\mb{W}^K -\mb{J}\big)\big(\nf(\x_t) - \mb{v}_{t-1}\big)\big\rangle \nonumber\\
=&~\lambda^2\| \mb{y}_{t}-\mb{J}\mb{y}_{t}\|^2
+ \lambda^2\E\big[\|\mb{v}_{t} - \mb{v}_{t-1}\|^2|\F_t\big]  
+ 2\big\langle\mb{W}^K\mb{y}_{t}-\mb{J}\mb{y}_{t},\big(\mb{W}^K -\mb{J}\big)\big(\nf(\x_{t-1}) - \mb{v}_{t-1}\big)\big\rangle
\nonumber\\
&+ 2\big\langle\mb{W}^K\mb{y}_{t}-\mb{J}\mb{y}_{t},\big(\mb{W}^K -\mb{J}\big)\big(\nf(\x_t) - \nf(\x_{t-1})\big)\big\rangle
\nonumber\\
=&~\lambda^2\| \mb{y}_{t}-\mb{J}\mb{y}_{t}\|^2
+ \lambda^2\E\big[\|\mb{v}_{t} - \mb{v}_{t-1}\|^2|\F_t\big]  
+ 2\big\langle\mb{W}^K\mb{y}_{t},\big(\mb{W}^K -\mb{J}\big)\big(\nf(\x_{t-1}) - \mb{v}_{t-1}\big)\big\rangle
\nonumber\\
&+ \underbrace{2\big\langle\mb{W}^K\mb{y}_{t}-\mb{J}\mb{y}_{t},\big(\mb{W}^K -\mb{J}\big)\big(\nf(\x_t) - \nf(\x_{t-1})\big)\big\rangle}_{=:A_t},
\end{align}
where the first line uses the fact that~$\x_t$, $\y_t$, $\mb{v}_{t-1}$ are $\F_t$-measurable and also Assumption~\ref{asp_unbias}, while the last line uses Lemma~\ref{lem_weight}\eqref{W_equal}. Towards the last term in~\eqref{SA_GT_0}, we observe that~$\forall t\geq2$ and $\forall\eta>0$,
\begin{align}\label{Term1}
A_t 
\leq&~2\big\|\mb{W}^K\mb{y}_{t}-\mb{J}\mb{y}_{t}\big\| \big\|\big(\mb{W}^K -\mb{J}\big)\big(\nf(\x_t) - \nf(\x_{t-1})\big)\big\|     \nonumber\\
\leq&~2\lambda\|\mb{y}_{t}-\mb{J}\mb{y}_{t}\| \lambda\|\nf(\x_t) - \nf(\x_{t-1})\| \nonumber\\
\leq&~2\lambda\|\mb{y}_{t}-\mb{J}\mb{y}_{t}\| \lambda L\|\x_t - \x_{t-1}\| \nonumber\\
\leq&~\eta\lambda^2\|\mb{y}_{t}-\mb{J}\mb{y}_{t}\|^2 + \eta^{-1}\lambda^2 L^2\|\x_t - \x_{t-1}\|^2,
\end{align}
where the first line uses the Cauchy-Schwarz inequality, the second line uses Lemma~\ref{lem_weight}, the third line uses the~$L$-smoothness of each~$f_i$, and last the line uses Young's inequality. Combining~\eqref{Term1} and~\eqref{SA_GT_0} leads to the following:~$\forall t\geq2$,
\begin{align}\label{SA_GT_1}
\E\big[\|\y_{t+1} -\mb{J}\y_{t+1}\|^2|\F_t\big] 
\leq&~(1+\eta)\lambda^2\| \mb{y}_{t}-\mb{J}\mb{y}_{t}\|^2
+ \eta^{-1}\lambda^2 L^2\|\x_t - \x_{t-1}\|^2 \nonumber\\
&+ \lambda^2\underbrace{\E\big[\|\mb{v}_{t} - \mb{v}_{t-1}\|^2|\F_t\big]}_{=:B_t}  
+ 2\underbrace{\big\langle\mb{W}^K\mb{y}_{t},\big(\mb{W}^K -\mb{J}\big)\big(\nf(\x_{t-1}) - \mb{v}_{t-1}\big)\big\rangle}_{=:C_t}.
\end{align}
In the following, we bound~$B_t$ and~$C_t$ in~\eqref{SA_GT_1} respectively.

\subsubsection{Step 2: Controlling~\texorpdfstring{$B_t$}{lg}}
We decompose~$B_t$ as follows:~$\forall t\geq2$,
\begin{align}\label{Bt_0}
B_t
=&~\E\big[\|\mb{v}_{t} - \nf(\x_t) + \nf(\x_t) - \mb{v}_{t-1}\|^2|\F_t\big]
\nonumber\\
=&~\E\big[\|\mb{v}_{t} - \nf(\x_t)\|^2|\F_t\big] 
+ \|\nf(\x_t) - \mb{v}_{t-1}\|^2 
\nonumber\\
\leq&~\E\big[\|\mb{v}_{t} - \nf(\x_t)\|^2|\F_t\big] 
+ 2\|\nf(\x_t) - \nf(\x_{t-1})\|^2 
+ 2\|\nf(\x_{t-1}) - \mb{v}_{t-1}\|^2 
\nonumber\\
\leq&~\E\big[\|\mb{v}_{t} - \nf(\x_t)\|^2|\F_t\big] 
+ 2L^2\|\x_t - \x_{t-1}\|^2 
+ 2\|\nf(\x_{t-1}) - \mb{v}_{t-1}\|^2, 
\end{align}
where the first line utilizes Assumption~\ref{asp_unbias} and the fact that $\nf(\x_t)$ and~$\mb{v}_{t-1}$ are $\F_t$-measurable, while the last line uses the~$L$-smoothness of each~$f_i$. To proceed, we note that~$\forall t\geq1$,
\begin{align}\label{Bt_01}
\E\big[\|\mb{v}_{t} - \nf(\x_t)\|^2|\F_t\big]    
=&~\sum_{i=1}^n\E\left[\left\|\frac{1}{b}\sum_{s=1}^b\n G_i(\x_t^i,\X_{i,s}^t) - \n f_i(\x_t^i)\right\|^2\bigg|\F_t\right]   \nonumber\\
=&~\frac{1}{b^2}\sum_{i=1}^n\sum_{s=1}^b\E\Big[\big\|\n G_i(\x_t^i,\X_{i,s}^t) - \n f_i(\x_t^i)\big\|^2\big|\F_t\Big] 
\leq \frac{n\nu^2}{b},
\end{align}
where the second line uses the fact that~$\x_t^i$ is~$\F_t$-measurable and~$\{\X_{i,1}^t,\cdots,\X_{i,b}^t,\F_t\}$ is an independent family for all $i\in\mc{V}$. Combining~\eqref{Bt_0} and~\eqref{Bt_01}, we conclude that
\begin{align}\label{Bt_bound}
\E\big[B_t\big]
\leq 
2L^2\E\big[\|\x_t - \x_{t-1}\|^2\big]
+ \frac{3n\nu^2}{b},
\qquad\forall t\geq2. 
\end{align}

\subsubsection{Step 3: Controlling~\texorpdfstring{$C_t$}{lg}}
Towards~$C_t$, we observe that~$\forall t\geq2$,
\begin{align}
\E[C_t|\F_{t-1}]\label{Ct_0}
=&~\E\Big[\Big\langle\mb{W}^{2K}\big(\mb{y}_{t-1} + \mb{v}_{t-1} - \mb{v}_{t-2}\big),\big(\mb{W}^K -\mb{J})(\nf(\x_{t-1}) - \mb{v}_{t-1}\big)\Big\rangle\big|\F_{t-1}\Big]   \nonumber\\
=&~\E\Big[\Big\langle\mb{W}^{2K}\mb{v}_{t-1},\big(\mb{W}^K -\mb{J}\big)\big(\nf(\x_{t-1}) - \mb{v}_{t-1}\big)\Big\rangle\big|\F_{t-1}\Big]
\nonumber\\
=&~\E\Big[\Big\langle\mb{W}^{2K}\big(\mb{v}_{t-1}-\nf(\x_{t-1})\big),\big(\mb{J}-\mb{W}^K\big)\big(\mb{v}_{t-1} - \nf(\x_{t-1})\big)\Big\rangle\big|\F_{t-1}\Big],
\end{align}
where the first line uses the~$\y$-update in Algorithm~\ref{PGT}, while the second and the last line use Assumption~\ref{asp_unbias} with the~$\F_{t-1}$-measurability of $\mb{y}_{t-1}$, $\mb{v}_{t-2}$ and $\nf(\x_{t-1})$. To proceed, note that for all~$t\geq1$ we have
\begin{align}\label{ip_indep}
\E\Big[\Big\langle \mb{v}_t^i - \n f_i(
\x_t^i), \mb{v}_t^r - \n f_i(
\x_t^r)\Big\rangle\big|\F_{t}\Big]  
= 0,
\end{align}
whenever~$i\neq r$. In light of~\eqref{ip_indep}, we proceed from~\eqref{Ct_0} as follows:~$\forall t\geq2$,
\begin{align}\label{Ct_1}
\E[C_t|\F_{t-1}]
=&~\E\Big[\big(\mb{v}_{t-1}-\nf(\x_{t-1})\big)^\top\big(\mb{J}-(\W^{K})^\top\mb{W}^{2K}\big)\big(\mb{v}_{t-1} - \nf(\x_{t-1})\big)\big|\F_{t-1}\Big], \nonumber\\
=&~\E\Big[\big(\mb{v}_{t-1}-\nf(\x_{t-1})\big)^\top\mbox{diag}\big(\J-(\W^{\top})^K\mb{W}^{2K}\big)\big(\mb{v}_{t-1} - \nf(\x_{t-1})\big)\big|\F_{t-1}\Big] \nonumber\\
\leq&~\E\Big[\big(\mb{v}_{t-1}-\nf(\x_{t-1})\big)^\top\mbox{diag}(\J)\big(\mb{v}_{t-1} - \nf(\x_{t-1})\big)\big|\F_{t-1}\Big]
\nonumber\\
=&~\frac{1}{n}\E\big[\|\mb{v}_{t-1} - \nf(\x_{t-1})\|^2|\F_{t-1}\big]
\nonumber\\
\leq&~\frac{\nu^2}{b},
\end{align}
where the first line uses Lemma~\ref{lem_weight}\eqref{W_equal}, the second line uses~\eqref{ip_indep}, the third line uses the entry-wise nonnegativity of~$\W$, and the last line uses~\eqref{Bt_01}.
Therefore, we conclude from~\eqref{Ct_1} that
\begin{align}\label{Ct_bound}
\E[C_t] \leq \frac{\nu^2}{b}, \qquad\forall t\geq2.  
\end{align}

\subsubsection{Step 4: Putting Bounds Together and Refining}
We substitute~\eqref{Bt_bound} and~\eqref{Ct_bound} into~\eqref{SA_GT_1} to obtain:~$\forall t\geq2$,
\begin{align}\label{SA_GT_2}
\E\big[\|\y_{t+1} -\mb{J}\y_{t+1}\|^2\big] 
\leq&~(1+\eta)\lambda^2\E\big[\| \mb{y}_{t}-\mb{J}\mb{y}_{t}\|^2\big]
+ (\eta^{-1}+2)\lambda^2 L^2\E\big[\|\x_t - \x_{t-1}\|^2\big] + (3\lambda^2n+2)\nu^2/b.
\end{align}
Setting~$\eta = \frac{1-\lambda^2}{2\lambda^2}$, we have:~$\forall t\geq2$,
\begin{align}\label{SA_GT_3}
\E\big[\|\y_{t+1} -\mb{J}\y_{t+1}\|^2\big] 
\leq&~\frac{1+\lambda^2}{2}\E\big[\| \mb{y}_{t}-\mb{J}\mb{y}_{t}\|^2\big]
+ \frac{2\lambda^2 L^2}{1-\lambda^2}\E\big[\|\x_t - \x_{t-1}\|^2\big] + \frac{(3\lambda^2n+2)\nu^2}{b}.
\end{align}
We then apply Lemma~\ref{lem_acc} to~\eqref{SA_GT_3} to obtain:~$\forall T\geq2$,
\begin{align}\label{SA_GT_5}
\sum_{t=2}^{T+1}\E\big[\|\y_{t} -\mb{J}\y_{t}\|^2\big] 
\leq&~\frac{2\E\big[\| \mb{y}_{2}-\mb{J}\mb{y}_{2}\|^2\big]}{1-\lambda^2}
+ \frac{4\lambda^2 L^2}{(1-\lambda^2)^2}\sum_{t=2}^{T}\E\big[\|\x_t - \x_{t-1}\|^2\big] + \frac{2(T-1)(3\lambda^2n+2)\nu^2}{b(1-\lambda^2)}. 
\end{align}
Since~$\y_1 = \mb{v}_0 = \mb{0}_{np}$, we have
\begin{align}\label{SA_GT_init}
\E\big[\| \mb{y}_{2}-\mb{J}\mb{y}_{2}\|^2\big]    
= \E\big[\|(\W^K-\J)\mb{v}_1\|^2\big]  
\leq \lambda^2\E\big[\|\mb{v}_1\|^2\big]  
=&~\lambda^2\|\nf(\x_1)\|^2
+ \lambda^2\E\big[\|\mb{v}_1 - \nf(\x_1)\|^2\big]  \nonumber\\
\leq&~\lambda^2\|\nf(\x_1)\|^2 + \lambda^2n\nu^2/b,
\end{align}
where the second line uses Lemma~\ref{lem_weight}\eqref{W_spectral}, the third line uses Lemma~\ref{asp_unbias}, and the last line is due to~\eqref{Bt_01}. Finally, we apply~\eqref{SA_GT_init} to~\eqref{SA_GT_5} to obtain:~$\forall T\geq2$,
\begin{align*}
\sum_{t=2}^{T+1}\E\big[\|\y_{t} -\mb{J}\y_{t}\|^2\big] 
\leq \frac{2\lambda^2n\zeta^2}{1-\lambda^2}
+ \frac{4\lambda^2 L^2}{(1-\lambda^2)^2}\sum_{t=2}^{T}\E\big[\|\x_t - \x_{t-1}\|^2\big] + \frac{2T(3\lambda^2n+2)\nu^2}{b(1-\lambda^2)}
+ \frac{2\lambda^2n\nu^2}{b(1-\lambda^2)}.
\end{align*}
The proof of Lemma~\ref{lem_GT}\eqref{lem_GT_M} follows by applying Lemma~\ref{lem_x_diff} to the above inequality with minor manipulations.

\subsection{Proof of Lemma~\ref{lem_GT}(\ref{lem_GT_O}) and~\ref{lem_GT}(\ref{lem_GT_E})}
We first establish a gradient tracking error bound that is applicable to both Algorithm~\ref{PGT_SR_O} and~\ref{PGT_SR_E}.
For ease of exposition, we denote 
\begin{align}\label{def_vr_sum}
\Upsilon_t := \|\mb{v}_t - \nf(\x_t) \|^2,
\qquad
\forall t\geq1.
\end{align}
\begin{lem}
Let Assumption~\ref{asp_mss} hold. Suppose that $T = Rq$ for some $R\in\mathbb{Z}^{+}$.
Consider the iterates generated by Algorithm~\ref{PGT_SR_O} or~\ref{PGT_SR_E}.  
Then we have: $\forall T\geq 2q$,
\begin{align}\label{GT_VR_V}
\sum_{t=2}^{T+1}\E\big[\|\y_{t} -\mb{J}\y_{t}\|^2\big] 
\leq&~\frac{2\lambda^2n\zeta^2}{1-\lambda^2}
+ \frac{96\lambda^2L^2}{(1-\lambda^2)^2}\sum_{t=1}^{T}\E\big[\|\x_t - \J\x_{t}\|^2\big]
+ \frac{48\lambda^2n\a^2L^2}{(1-\lambda^2)^2}\sum_{t=1}^{T-1}\E\big[\|\ol{\sgp}_{t}\|^2\big]
\nonumber\\
&+ \frac{14\lambda^2}{(1-\lambda^2)^2} \sum_{z=0}^{R-1}\E\big[\Upsilon_{zq+1}\big].
\end{align}
\end{lem}
\begin{proof}
We first recall from~\eqref{GT_0} that $\forall t\geq1$,
\begin{align}\label{GT_VR_0}
\|\y_{t+1} -\mb{J}\y_{t+1}\|^2 
\leq \lambda^2\| \mb{y}_{t}-\mb{J}\mb{y}_{t}\|^2
+ \lambda^2\|\mb{v}_{t} - \mb{v}_{t-1}\|^2 
+ 2\big\langle\mb{W}^K\mb{y}_{t}-\mb{J}\mb{y}_{t},\big(\mb{W}^K -\mb{J}\big)\big(\mb{v}_{t} - \mb{v}_{t-1}\big)\big\rangle.
\end{align}
In the first two steps, we refine~\eqref{GT_VR_0} for $\bmod(t,q) \neq1$ and $\bmod(t,q) = 1$ respectively. 

\vspace{0.1cm}
\noindent
\textbf{Step 1: consider any $t\geq2$ such that $\bmod(t,q) \neq1$.}
From~\eqref{GT_VR_0}, we observe that for all $\eta>0$,
\begin{align}\label{GT_VR_I_0}
\|\y_{t+1} -\mb{J}\y_{t+1}\|^2 
\leq&~\lambda^2\| \mb{y}_{t}-\mb{J}\mb{y}_{t}\|^2
+ \lambda^2\|\mb{v}_{t} - \mb{v}_{t-1}\|^2 
+ 2\big\|\mb{W}^K\mb{y}_{t}-\mb{J}\mb{y}_{t}\big\|\big\|\big(\mb{W}^K -\mb{J}\big)\big(\mb{v}_{t} - \mb{v}_{t-1}\big)\big\| 
\nonumber\\
\leq&~\lambda^2\| \mb{y}_{t}-\mb{J}\mb{y}_{t}\|^2
+ \lambda^2\|\mb{v}_{t} - \mb{v}_{t-1}\|^2 
+ 2\lambda^2\big\|\mb{y}_{t}-\mb{J}\mb{y}_{t}\big\|\big\|\mb{v}_{t} - \mb{v}_{t-1}\big\| \nonumber\\
\leq&~\lambda^2(1+\eta)\| \mb{y}_{t}-\mb{J}\mb{y}_{t}\|^2
+ \lambda^2(1+\eta^{-1})\|\mb{v}_{t} - \mb{v}_{t-1}\|^2,
\end{align}
where the first line uses Cauchy-Schwarz inequality, the second line uses Lemma~\ref{lem_weight}, and the last uses Young's inequality.
Setting $\eta = \frac{1-\lambda^2}{2\lambda^2}$ in~\eqref{GT_VR_I_0} gives: 
\begin{align}\label{GT_VR_I_1}
\|\y_{t+1} -\mb{J}\y_{t+1}\|^2 
\leq \frac{1+\lambda^2}{2}\| \mb{y}_{t}-\mb{J}\mb{y}_{t}\|^2
+ \frac{\lambda^2(1+\lambda^2)}{1-\lambda^2}\|\mb{v}_{t} - \mb{v}_{t-1}\|^2
\end{align}
Note that
\begin{align}\label{vt_diff_I}
\E\big[\|\mb{v}_{t} - \mb{v}_{t-1}\|^2\big] 
=&~\sum_{i=1}^n\E\!\left[\left\|\frac{1}{b}\sum_{s=1}^{b}\Big(\n G_i(\x_t^i,\X_{i,s}^t) - \n G_i(\x_{t-1}^i,\X_{i,s}^t)\Big)\right\|^2\right]
\nonumber\\
\leq&~\frac{1}{b}\sum_{i=1}^n\sum_{s=1}^{b}\E\!\left[\left\|\n G_i(\x_t^i,\X_{i,s}^t) - \n G_i(\x_{t-1}^i,\X_{i,s}^t)\right\|^2\right]
\nonumber\\
\leq&~L^2\E\big[\|\x_t - \x_{t-1}\|^2\big],
\end{align}
where the last line uses the mean-squared smoothness.
Applying~\eqref{vt_diff_I} to~\eqref{GT_VR_I_1}, we obtain:
\begin{align}\label{GT_VR_I}
\E\big[\|\y_{t+1} -\mb{J}\y_{t+1}\|^2\big]
\leq \frac{1+\lambda^2}{2}\E\big[\| \mb{y}_{t}-\mb{J}\mb{y}_{t}\|^2\big]
+ \frac{2\lambda^2L^2}{1-\lambda^2}\E\big[\|\mb{x}_{t} - \mb{x}_{t-1}\|^2\big]
\end{align}

\noindent
\textbf{Step 2: consider any $t\geq2$ such that $\bmod(t,q) = 1$.}
In this case, we have $\E[\mb{v}_t|\F_t] = \nf(\x_t)$.
Taking the conditional expectation of~\eqref{GT_VR_0} with respect to the filtration $\F_t$, we obtain
\begin{align}\label{GT_VR_II_0}
\E\big[\|\y_{t+1} -\mb{J}\y_{t+1}\|^2|\F_t\big] 
\leq&~\lambda^2\| \mb{y}_{t}-\mb{J}\mb{y}_{t}\|^2
+ \lambda^2\E\big[\|\mb{v}_{t} - \mb{v}_{t-1}\|^2|\F_t\big]
\nonumber\\
&+ 2\big\langle\mb{W}^K\mb{y}_{t}-\mb{J}\mb{y}_{t},\big(\mb{W}^K -\mb{J}\big)\big(\nf(\x_t) - \mb{v}_{t-1}\big)\big\rangle
\nonumber\\
\leq&~\lambda^2(1+\eta)\| \mb{y}_{t}-\mb{J}\mb{y}_{t}\|^2
+ \lambda^2\E\big[\|\mb{v}_{t} - \mb{v}_{t-1}\|^2|\F_t\big]
+ \lambda^2\eta^{-1}\|\nf(\x_t) - \mb{v}_{t-1}\|^2,
\end{align}
where the first line uses the fact that~$\x_t$ and $\y_t$ are $\F_t$-measurable and the second line follows a similar line of arguments as in~\eqref{GT_VR_I_0}. Setting~$\eta = \frac{1-\lambda^2}{2\lambda^2}$ in~\eqref{GT_VR_II_0}, we obtain:
\begin{align}\label{GT_VR_II_1}
\E\big[\|\y_{t+1} -\mb{J}\y_{t+1}\|^2\big] 
\leq
\frac{1+\lambda^2}{2}\E\big[\| \mb{y}_{t}-\mb{J}\mb{y}_{t}\|^2\big]
+ \lambda^2\E\big[\|\mb{v}_{t} - \mb{v}_{t-1}\|^2\big]
+ \frac{2\lambda^4}{1-\lambda^2}\E\big[\|\nf(\x_t) - \mb{v}_{t-1}\|^2\big].
\end{align}
We recall the definition of $\Upsilon_t$ in~\eqref{def_vr_sum} and  observe that
\begin{align}\label{vt_diff_II}
\|\mb{v}_{t} - \mb{v}_{t-1}\|^2
=&~\|\mb{v}_{t} - \nf(\x_t) + \nf(\x_t) - \nf(\x_{t-1}) + \nf(\x_{t-1}) - \mb{v}_{t-1}\|^2 \nonumber\\
\leq&~3\Upsilon_t + 3\|\nf(\x_t) - \nf(\x_{t-1})\|^2  + 3\Upsilon_{t-1}
\nonumber\\
\leq&~3\Upsilon_t + 3L^2\|\x_t - \x_{t-1}\|^2  + 3\Upsilon_{t-1},
\end{align}
where the last uses the $L$-smoothness of each $f_i$. Similarly, we have
\begin{align}\label{GT_VR_II_1_aux}
\|\nf(\x_t) - \mb{v}_{t-1}\|^2
=&~\|\nf(\x_t) - \nf(\x_{t-1}) + \nf(\x_{t-1})  - \mb{v}_{t-1}\|^2
\nonumber\\
\leq&~2L^2\|\x_t - \x_{t-1}\|^2 + 2\Upsilon_{t-1}.
\end{align}
Plugging~\eqref{vt_diff_II} and~\eqref{GT_VR_II_1_aux} into~\eqref{GT_VR_II_1} gives
\begin{align}\label{GT_VR_II}
\E\big[\|\y_{t+1} -\mb{J}\y_{t+1}\|^2\big] 
\leq&~\frac{1+\lambda^2}{2}\E\big[\| \mb{y}_{t}-\mb{J}\mb{y}_{t}\|^2\big]
+ \left(3\lambda^2+ \frac{4\lambda^4}{1-\lambda^2}\right)
L^2\E\big[\|\x_{t} - \mb{x}_{t-1}\|^2\big]
\nonumber\\
&
+ 3\lambda^2\E\big[\Upsilon_{t}\big]
+ \left(3\lambda^2+ \frac{4\lambda^4}{1-\lambda^2}\right)
\E\big[\Upsilon_{t-1}\big]
 \nonumber\\
\leq&~\frac{1+\lambda^2}{2}\E\big[\| \mb{y}_{t}-\mb{J}\mb{y}_{t}\|^2\big]
+ \frac{4\lambda^2L^2}{1-\lambda^2}\E\big[\|\x_{t} - \mb{x}_{t-1}\|^2\big]
+ 3\lambda^2\E\big[\Upsilon_{t}\big]
+ \frac{4\lambda^2\E\big[\Upsilon_{t-1}\big]}{1-\lambda^2}.
\end{align}

\noindent
\textbf{Step 3: combining step 1 and step 2.}
Combining~\eqref{GT_VR_I} and~\eqref{GT_VR_II}, we obtain: $\forall t\geq2$,
\begin{align}\label{GT_VR_III_0}
\E\big[\|\y_{t+1} -\mb{J}\y_{t+1}\|^2\big] 
\leq&~\frac{1+\lambda^2}{2}\E\big[\| \mb{y}_{t}-\mb{J}\mb{y}_{t}\|^2\big]
+ \frac{4\lambda^2L^2}{1-\lambda^2}\E\big[\|\x_{t} - \mb{x}_{t-1}\|^2\big]
\nonumber\\
&
+ \i_{\{\!\bmod(t,q)=1\}}\left(3\lambda^2\E\big[\Upsilon_t\big]
+ \frac{4\lambda^2\E[\Upsilon_{t-1}]}{1-\lambda^2}\right).
\end{align}
Let~$T = Rq$ for some $R\in\mathbb{Z}^{+}$. We apply Lemma~\ref{lem_acc} to~\eqref{GT_VR_III_0} to obtain: $\forall T\geq 2q$,
\begin{align}\label{GT_VR_III_1}
&\sum_{t=2}^{T+1}\E\big[\|\y_{t} -\mb{J}\y_{t}\|^2\big] \nonumber\\
\leq&~\frac{2\E\big[\|\mb{y}_{2}-\mb{J}\mb{y}_{2}\|^2\big]}{1-\lambda^2}
+ \frac{8\lambda^2L^2}{(1-\lambda^2)^2}\sum_{t=2}^{T}\E\big[\|\x_{t} - \mb{x}_{t-1}\|^2\big]
+ \sum_{t=2}^{T}\i_{\{\!\bmod(t,q)=1\}}\left(\frac{6\lambda^2\E[\Upsilon_t]}{1-\lambda^2}
+ \frac{8\lambda^2\E[\Upsilon_{t-1}]}{(1-\lambda^2)^2}\right)
\nonumber\\
=&~\frac{2\E\big[\|\mb{y}_{2}-\mb{J}\mb{y}_{2}\|^2\big]}{1-\lambda^2}
+ \frac{8\lambda^2L^2}{(1-\lambda^2)^2}\sum_{t=2}^{T}\E\big[\|\x_{t} - \mb{x}_{t-1}\|^2\big]
+ \sum_{z=1}^{R-1}\left(\frac{6\lambda^2}{1-\lambda^2}\E\big[\Upsilon_{zq+1}\big]
+ \frac{8\lambda^2}{(1-\lambda^2)^2}\E\big[\Upsilon_{zq}\big]\right).
\end{align}
Note that
\begin{align}\label{VR_GT_init}
\E\big[\| \mb{y}_{2}-\mb{J}\mb{y}_{2}\|^2\big]    
= \E\big[\|(\W^K-\J)\mb{v}_1\|^2\big]  
\leq \lambda^2\E\big[\|\mb{v}_1\|^2\big]  
= \lambda^2\|\nf(\x_1)\|^2
+ \lambda^2\E\big[\Upsilon_1\big]
\end{align}
Applying~\eqref{VR_GT_init} to~\eqref{GT_VR_III_1} gives the following: $\forall T\geq 2q$,
\begin{align}\label{GT_VR_III}
&\sum_{t=2}^{T+1}\E\big[\|\y_{t} -\mb{J}\y_{t}\|^2\big] \nonumber\\
\leq&~\frac{2\lambda^2n\zeta^2}{1-\lambda^2}
+ \frac{8\lambda^2L^2}{(1-\lambda^2)^2}\sum_{t=2}^{T}\E\big[\|\x_{t} - \mb{x}_{t-1}\|^2\big] 
+ \frac{2\lambda^2}{1-\lambda^2}\E[\Upsilon_1]
+ \frac{6\lambda^2}{1-\lambda^2}\sum_{z=1}^{R-1}\E\big[\Upsilon_{zq+1}\big]
+ \frac{8\lambda^2}{(1-\lambda^2)^2}\sum_{z=1}^{R-1}\E\big[\Upsilon_{zq}\big]
\nonumber\\
\leq&~\frac{2\lambda^2n\zeta^2}{1-\lambda^2}
+ \frac{8\lambda^2L^2}{(1-\lambda^2)^2}\sum_{t=2}^{T}\E\big[\|\x_{t} - \mb{x}_{t-1}\|^2\big] 
+ \frac{6\lambda^2}{1-\lambda^2}\sum_{z=0}^{R-1}\E\big[\Upsilon_{zq+1}\big]
+ \frac{8\lambda^2}{(1-\lambda^2)^2}\sum_{z=1}^{R-1}\E\big[\Upsilon_{zq}\big]
\end{align}

\noindent
\textbf{Step 4: bounding $\E\big[\Upsilon_t\big]$.}
The derivations in this step essentially repeat the proof of Lemma~\ref{lem_VR_SARAH}.
Recall the definition of $\Upsilon_t$ in~\eqref{def_vr_sum}.
Consider any $t\geq1$ such that $\bmod(t,q) \neq 1$ and define:~$\forall i\in\mc{V}$,
\begin{align*}
\mb{d}_t^{i,s} :
= \n G_i(\x_t^i,\X_{i,s}^t) - \n G_i(\x_{t-1}^i,\X_{i,s}^t),
\qquad
\mb{d}_t^{i} := \frac{1}{b}\sum_{s=1}^{b}\mb{d}_t^{i,s}.
\end{align*}
By the update of~$\mb{v}_t^i$, we observe that
\begin{align}\label{vs1}
\E\big[\Upsilon_t|\F_t\big]  
=&~\sum_{i=1}^n\E\!\left[\left\|\mb{d}_t^{i}+\mb{v}_{t-1}^i - \n f_i(\x_t^i)\right\|^2\big|\F_t\right]   \nonumber\\
=&~\sum_{i=1}^n\E\!\left[\left\|\mb{d}_t^{i} - \n f_i(\x_t^i) + \n f_i(\x_{t-1}^i)  +\mb{v}_{t-1}^i - \n f_i(\x_{t-1}^i)\right\|^2 \big|\F_t\right] \nonumber\\
=&~\sum_{i=1}^n\E\!\left[\left\|\mb{d}_t^{i} - \n f_i(\x_t^i) + \n f_i(\x_{t-1}^i)\right\|^2 \big|\F_t\right] 
+ \Upsilon_{t-1}
\nonumber\\
=&~\frac{1}{b^2}\sum_{i=1}^n\sum_{s=1}^b\E\!\left[\big\|\mb{d}_t^{i,s} - \n f_i(\x_t^i) + \n f_i(\x_{t-1}^i)\big\|^2 \big|\F_t\right] 
+ \Upsilon_{t-1}
\nonumber\\
\leq&~\frac{1}{b^2}\sum_{i=1}^n\sum_{s=1}^b\E\!\left[\big\|\mb{d}_t^{i,s}\big\|^2 \big|\F_t\right] 
+ \Upsilon_{t-1},
\end{align}
where the above derivations follow a very similar line of arguments as in~\eqref{v1} and thus we omit the detailed explanations. Taking the expectation of~\eqref{vs1} and using the mean-squared smoothness of $\n G(\cdot,\bxi)$, we obtain
\begin{align}\label{vs11}
\E\big[\Upsilon_t\big]
\leq \frac{L^2}{b}\E\big[\|\x_t - \x_{t-1}\|^2\big]
+ \E\big[\Upsilon_{t-1}\big].
\end{align}
For convenience, we define
$\vp_t 
:= \left\lfloor \frac{t-1}{q} \right\rfloor,
\forall t\geq1.$
It can be verified that
$\vp_t q + 1 \leq t \leq (\vp_t + 1)q,
\forall t\geq1.$
Recursively applying~\eqref{vs11} from $t$ to $(\varphi_t+1)$, we obtain
\begin{align}\label{vs2}
\E\big[\Upsilon_t\big]  
\leq
\frac{L^2}{b}\sum_{j=\varphi_t+1}^{t}\E\big[\|\x_j - \x_{j-1}\|^2\big] + \E\big[\Upsilon_{\varphi_t}\big].
\end{align}
In particular, taking $t = zq$ for some $z\in\mathbb{Z}^{+}$ in~\eqref{vs2} gives
\begin{align}\label{vs3}
\E\big[\Upsilon_{zq}\big]  
\leq
\frac{L^2}{b}\sum_{j=(z-1)q+2}^{zq}\E\big[\|\x_j - \x_{j-1}\|^2\big] + \E\big[\Upsilon_{(z-1)q+1}\big].
\end{align}
We sum up~\eqref{vs3} over~$z$ from $1$ to $R$ to obtain: 
\begin{align}\label{vs4}
\sum_{z=1}^{R}\E\big[\Upsilon_{zq}\big]  
\leq&~\frac{L^2}{b}\sum_{z=1}^{R}\sum_{j=(z-1)q+2}^{zq}\E\big[\|\x_j - \x_{j-1}\|^2\big] + \sum_{z=1}^{R}\E\big[\Upsilon_{(z-1)q+1}\big]
\nonumber\\
\leq&~\frac{L^2}{b}\sum_{t=2}^{T}\E\big[\|\x_t - \x_{t-1}\|^2\big] + \sum_{z=0}^{R-1}\E\big[\Upsilon_{zq+1}\big].
\end{align}
\noindent
\textbf{Step 5: putting bounds together.} Applying~\eqref{vs4} to~\eqref{GT_VR_III} gives the following: $\forall T\geq 2q$,
\begin{align}\label{GT_VR_IV}
&\sum_{t=2}^{T+1}\E\big[\|\y_{t} -\mb{J}\y_{t}\|^2\big] 
\nonumber\\
\leq&~\frac{2\lambda^2n\zeta^2}{1-\lambda^2}
+ \left(1+\frac{1}{b}\right)\frac{8\lambda^2L^2}{(1-\lambda^2)^2}\sum_{t=2}^{T}\E\big[\|\x_t - \x_{t-1}\|^2\big]
+ \frac{6\lambda^2}{1-\lambda^2}\sum_{z=0}^{R-1}\E\big[\Upsilon_{zq+1}\big]
+ \frac{8\lambda^2}{(1-\lambda^2)^2} \sum_{z=0}^{R-1}\E\big[\Upsilon_{zq+1}\big]
\nonumber\\
\leq&~\frac{2\lambda^2n\zeta^2}{1-\lambda^2}
+ \frac{16\lambda^2L^2}{(1-\lambda^2)^2}\sum_{t=2}^{T}\E\big[\|\x_t - \x_{t-1}\|^2\big]
+ \frac{14\lambda^2}{(1-\lambda^2)^2} \sum_{z=0}^{R-1}\E\big[\Upsilon_{zq+1}\big].
\end{align}
Plugging Lemma~\ref{lem_x_diff} into~\eqref{GT_VR_IV} finishes the proof.
\end{proof}

\subsubsection{Proof of Lemma~\ref{lem_GT}(\ref{lem_GT_O})}
Lemma~\ref{lem_GT}(\ref{lem_GT_O}) follows from~\eqref{GT_VR_V} by $\Upsilon_{zq+1} \leq n\nu^2/B$ for all $z\in\mathbb{Z}^+$.
\subsubsection{Proof of Lemma~\ref{lem_GT}(\ref{lem_GT_E})}
Lemma~\ref{lem_GT}(\ref{lem_GT_E}) follows from~\eqref{GT_VR_V} by $\Upsilon_{zq+1} = 0$ for all $z\in\mathbb{Z}^+$.

\end{document}